\documentclass{article}
\usepackage{amsmath, amsthm, epsfig,amssymb, multicol, tikz}
\usepackage{listings}

\newcommand{\F}{\mathcal{F}}

\newcommand{\A}{\mathcal{A}}
\newcommand{\B}{\mathcal{B}}

\newcommand{\D}{\mathcal{D}}
\newcommand{\V}{\mathcal{V}}
\newcommand{\N}{\mathcal{N}}
\newcommand{\aH}{\mathcal{H}}

\newcommand{\Pa}{\mathcal{P}}
\newcommand{\nchn}{\binom{n}{\lfloor \frac{n}{2}\rfloor}}
\newcommand{\mchm}{\binom{m}{\lfloor \frac{m}{2}\rfloor}}

\newcommand{\ex}{{\rm ex}}
\newcommand{\Oh}{\mathcal{O}}
\newcommand{\La}{{\rm La}}
\newcommand{\E}{{\rm E}}
\renewcommand{\Pr}{{\rm Pr}}
\newtheorem{lemma}{Lemma}
\newtheorem{theorem}{Theorem}
\newtheorem{corollary}{Corollary}
\newtheorem{definition}{Definition}
\newtheorem{conjecture}{Conjecture}
\newtheorem{prop}{Proposition}

\def\nchn{{n \choose \lfloor \frac{n}{2}\rfloor}}

\title{Tur\'an Problems on Non-uniform Hypergraphs}
\author{
Travis Johnston \thanks{University of South Carolina, Columbia, SC 29208,
({\tt j.travis.johnston@gmail.com}).} \and 
Linyuan Lu
\thanks{University of South Carolina, Columbia, SC 29208,
({\tt lu@math.sc.edu}). This author was supported in part by NSF
grant  DMS 1000475. }
}

\begin{document}
\maketitle

\begin{abstract}
A non-uniform hypergraph $H=(V,E)$ consists of a vertex set $V$ and
an edge set $E\subseteq 2^V$; 
the edges in $E$ are not required to all have the same cardinality.
The set of all cardinalities of edges in $H$ is denoted by $R(H)$, the set of edge types.
For a fixed hypergraph $H$, the Tur\'an density $\pi(H)$ is defined to be 
$\lim_{n\to\infty}\max_{G_n}h_n(G_n)$,
where the maximum is taken over all $H$-free hypergraphs $G_n$ on $n$
vertices satisfying $R(G_n)\subseteq R(H)$,  and $h_n(G_n)$, the so called
Lubell function,  is the expected number of edges in $G_n$ hit by a random full chain.
This concept, which generalizes  the Tur\'an density of $k$-uniform hypergraphs,
 is motivated by recent work on extremal poset problems. 
The details connecting these two areas will be revealed in the end of this paper.

Several properties of Tur\'an density, such as supersaturation,
blow-up, and suspension, are generalized from uniform hypergraphs to non-uniform hypergraphs. 
Other questions such as ``Which hypergraphs are degenerate?" are more complicated and don't appear to generalize well.
In addition, we completely determine the Tur\'an densities of $\{1,2\}$-hypergraphs.
\end{abstract}

\section{Introduction}
A hypergraph $H$ is a pair $(V,E)$; $V$ is the vertex set, and
$E\subseteq 2^V$ is the edge set.  If all edges have the same
cardinality $k$, then $H$ is a $k$-uniform hypergraph.  Tur\'an
problems on $k$-uniform hypergraphs have been actively studied for
many decades.  However, Tur\'an problems on non-uniform hypergraphs
are rarely considered (see \cite{MZ, Lemons} for two different
treatments).   
Very recently, several groups of people have
started actively studying extremal families of sets avoiding given
sub-posets.  Several new problems have been established.  One of them
is the diamond problem:

{\bf The diamond conjecture: \cite{GriLu}} Any family $\F$ of subsets of
$[n]:=\{1,2,\ldots,n\}$ with no four sets $A,B,C,D$ satisfying
$A\subseteq B\cap C$, $B\cup C \subseteq D$
can have at most $(2+o(1))\nchn$ subsets.

This conjecture, along with many other problems, motivates us to study
Tur\'an-type problems on non-uniform hypergraphs.  
The details of this connection will be given in the last section.

We briefly review the history of Tur\'an Problems on uniform
hypergraphs.  Given a positive integer $n$ and a $k$-uniform
hypergraph $H$ on $n$ vertices (or $k$-graph, for short), the Tur\'an
number $\ex(n,H)$ is the maximum number of edges in a $k$-graph on
$n$ vertices that does not contain $H$ as a subgraph; such a graph is
called $H$-\emph{free}.  Katona et al. \cite{KNS} showed that
$f(n,H)=\ex(n,H)/{n\choose k}$ is a decreasing function of $n$.  The
limit $\displaystyle \pi(H)=\lim_{n\to \infty} f(n,H)$, which always exists, is
called the {\em Tur\'an density} of $H$.

For $k=2$, the graph case, Erd\H{o}s-Stone-Simonovits proved
$\pi(G)=1-\frac{1}{\chi(G)-1}$ for any graph $G$ with chromatic number
$\chi(G)\geq 3$.  If $G$ is bipartite, then $\ex(n,G)=o(n^2)$.  The
magnitude of $\ex(n,G)$ is unknown for most bipartite graphs $G$.

Let $K^r_k$ denote the complete $r$-graph on $k$ vertices.
Tur\'an determined the value of $\ex(n,K^2_{k})$ which 
implyies that $\pi(K^2_k)=1-\frac{1}{k-1}$ for all $k\geq 3$.
However, no Tur\'an density $\pi(K^r_k)$ is known for any $k>r\geq 3$.  
The most extensively studied case is when $k=4$ and $r=3$.
Tur\'an conjectured \cite{Tu} that $\pi(K_4^3)= 5/9$.
Erd\H{o}s \cite{Er81} offered \$500 for determining any $\pi(K^r_k)$
with $k>r\geq 3$ and \$1000 for answering it for all $k$ and $r$.
The upper bounds for $\pi(K_4^3)$ have been sequentially improved:
$0.6213$ (de Caen \cite{dC94}), $0.5936$ (Chung-Lu \cite{ChLu}),
$0.56167$ (Razborov \cite{Razborov}, using flag
algebra method.)  There are a few uniform hypergraphs whose Tur\'an
density has been determined: the Fano plane \cite{FS05,KS05a},
expanded triangles \cite{KS05b}, $3$-books, $4$-books \cite{FMP},
 $F_5$ \cite{FrFu83}, extended complete graphs \cite{Pik}, etc.
In particular, Baber \cite{baber} recently found the Tur\'an density
of many $3$-uniform hypergraphs using flag algebra methods.
For a more complete survey of methods and results on uniform hypergraphs see Peter Keevash's survey paper \cite{KeevashSurvey}.

A non-uniform hypergraph $H=(V,E)$ consists of a vertex set $V$ and
an edge set $E\subseteq 2^V$. Here the edges of $E$ could have
different cardinalities. The set of all the cardinalities of edges in $H$ is 
denoted by $R(H)$, the set of edge types. For a fixed hypergraph $H$,
the Tur\'an density $\pi(H)$ is defined to be 
$\lim_{n\to\infty}\max_{G_n}h_n(G_n)$,
where the maximum is taken over all $H$-free hypergraphs $G_n$ on $n$
vertices satisfying $R(G_n)\subseteq R(H)$.
$h_n(G_n)$, the so called Lubell function,  is the
expected number of edges in $G_n$ hit by a random full chain.
The Lubell function has been a very useful tool in extremal poset theory,
in particular it has been used in the study of the diamond conjecture.

In section 2, we show that our notion of $\pi(H)$ is well-defined
and is consistent with the usual definition for uniform hypergraphs.
We also give examples of Tur\'an densities for several non-uniform hypergraphs.
In section 3, we generalize the supersaturation Lemma to non-uniform hypergraphs.
Then we prove that blowing-up will not affect the Tur\'an density. 
Using various techniques, we determine the Tur\'an density of every
$\{1,2\}$-hypergraph in section 4. 
Remarkably, the Tur\'an densities of $\{1,2\}$-hypergraphs are in the set
$$\bigg\{1,\frac{9}{8}, \frac{5}{4}, \frac{3}{2}, \frac{5}{3},
\ldots, 2-\frac{1}{k},\ldots \bigg \}.$$

Among $r$-uniform hypergraphs, $r$-partite hypergraphs have the smallest possible Tur\'an density.
Erd\H{o}s proved that any $r$-uniform hypergraph forbidding
the complete $r$-uniform $r$-partite hypergraphs can have
at most $O(n^{r-1/\delta})$ edges. We generalize this theorem 
to non-uniform hypergraphs. 
A hypergraph is degenerate if it has the smallest possible Tur\'an density.
For $r$-uniform hypergraphs, a hypergraph $H$ is degenerate if and only if $H$ is the subgraph of a blow-up of a single edge.
Unlike the degenerate $r$-uniform hypergraphs, the degenerate non-uniform hypergraphs are not classified. 
For non-uniform hypergraphs, chains--one natural extension of a single edge--are degenerate.
Additionally, every subgraph of a blow-up of a chain is also degenerate.
However, we give an example of a degenerate, non-uniform hypergraph not contained in any blow-up of a chain.
This leaves open the question of which non-uniform hypergraphs are degenerate.

In section 6, we consider the suspension of hypergraphs.
The suspension of a hypergraph $H$ is a new hypergraph, denoted by
$S(H)$, with one additional vertex, $\ast$, added to every edge of $H$.
In a hypergraph Tur\'an problem workshop hosted by 
the AIM Research Conference Center in 2011, the following conjecture
was posed: $\displaystyle \lim_{t\to\infty}\pi(S^t(K^{r}_n))=0$. We conjecture
$\displaystyle\lim_{t\to\infty} \pi(S^t(H))=|R(H)|-1$ holds for any hypergraph $H$.
Some partial results are proved.
Finally in the last section, we will point out the relation between
the Tur\'an problems on hypergraphs and extremal poset problems.

\section{Non-uniform hypergraphs}
\subsection{Notation}
Recall that a hypergraph $H$ is a pair $(V,E)$ with the vertex set $V$ and edge set $E\subseteq 2^{V}$.
Here we place no restriction on the cardinalities of edges.
The set $R(H):=\{|F|\colon F\in E\}$ is called the set of its {\em edge types}.
A hypergraph $H$ is $k$-uniform if $R(H)=\{k\}$.
It is non-uniform if it has at least two edge types.
For any $k\in R(H)$, the {\em level hypergraph} $H^k$ is the hypergraph consisting of all $k$-edges of $H$.
A uniform hypergraph $H$ has only one (non-empty) level graph, i.e., $H$ itself.
In general, a non-uniform hypergraph $H$ has $|R(H)|$ (non-empty) level hypergraphs.
Throughout the paper, for any finite set $R$ of non-negative integers, we say, $G$ is an $R$-graph if $R(G)\subseteq R$.
We write $G^R_n$ for a hypergraph on $n$ vertices with $R(G)\subseteq R$.
We simplify it to $G$ if $n$ and $R$ are clear under context.

Let $R$ be a fixed set of edge types.
Let $H$ be an $R$-graph.
The number of vertices in $H$ is denoted by $v(H):=|V(H)|$.
Our goal is to measure the edge density of $H$ and be able to compare it (in a meaningful way) to the edge density of other $R$ graphs.
The standard way to measure this density would be:
\[\mu(H) = \frac{|E(H)|}{\sum_{k\in R(H)}\binom{v(H)}{k}}.\]
This density ranges from 0 to 1 (as one would expect)--a complete $R$-graph having a density of 1.
Unfortunately, this is no longer a useful measure of density since the number of edges with maximum cardinality will dwarf the number of edges of all other sizes.
Specifically, one could take $k$-uniform hypergraph (where $k=\max\{r:r\in R(H)\}$) on enough vertices and make its density as close to 1 he likes.
The problem is that this $k$-uniform hypergraph is quite different from the complete $R$-graph (when $|R|>1$) with the same number of vertices.
Instead, we use the Lubell function to measure the edge density.
This is adapted from the use of the Lubell function studying families of subsets.

For a non-uniform hypergraph $G$ on $n$ vertices, we define  the Lubell function of $G$ as 
\begin{equation} \label{eq:lubell}
	h_{n}(G):=\sum_{F\in E(G)}\frac{1}{\binom{n}{|F|}}=\sum_{k\in R(G)}\frac{|E(H^{k})|}{\binom{n}{k}}.   
\end{equation}

The Lubell function is the expected number of edges hit by a random full chain.
Namely, pick a uniformly random permutation $\sigma$ on $n$ vertices; 
define a random full chain $C_\sigma$ by
\[\{ \{\emptyset\}, \{\sigma(1)\}, \{\sigma(1), \sigma(2)\}, \cdots, [n]\}.\]
Let $X=|E(G)\cap C_{\sigma}|$, the number of edges hit by the random full chain. 
Then
\begin{equation} \label{eq:X}
  h_n(G)=\E(X).
\end{equation}

Given two hypergraphs $H_1$ and $H_2$, we say $H_1$ is a subgraph of $H_2$, denoted by $H_1\subseteq H_2$, 
if there exists a 1-1 map $f\colon V(H_1)\to V(H_2)$ so that $f(F)\in E(H_2)$ for any $F\in E(H_1)$. 
Whenever this occurs, we say the image $f(H_1)$ is an {\em ordered copy} of $H_2$, written as $H_1\stackrel{f}{\hookrightarrow} H_2$.
A necessary condition for $H_1\subseteq H_2$ is $R(H_1)\subseteq R(H_2)$.
Given a subset $K\subseteq V(H)$ and a subset $S\subseteq R(H)$, the {\em induced subgraph}, 
denoted by $H^{[S]}[K]$,  is a hypergraph on $K$ with the edge set $\{F\in E(H)\colon F\subseteq K \mbox{ and } |F|\in S\}$.
When $S=R(H)$, we simply write $H[K]$ for $H^{[S]}[K]$.

Given a positive integer $n$ and  a subset $R\subseteq [n]$, 
the complete hypergraph $K^R_{n}$ is a hypergraph on $n$ vertices with edge set $\bigcup_{i\in R} \binom{[n]}{i}$.
For example, $K^{\{k\}}_n$ is the complete $k$-uniform hypergraph.
$K^{[k]}_n$ is the non-uniform hypergraph with all possible edges of cardinality at most $k$.

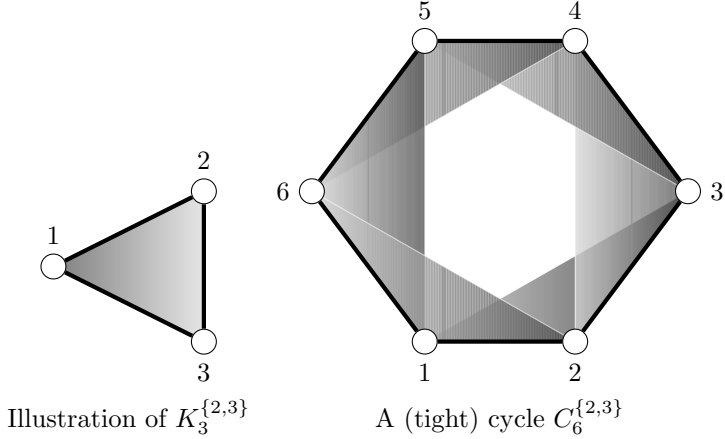
\begin{center}
\begin{tikzpicture}[vertex_open/.style={circle, draw=black, fill=white}, vertex_closed/.style={circle, draw=black, fill=black}] 
\shadedraw[left color=black!50!white, right color=black!10!white, draw=black!10!white] (0,0)--(2,1)--(2,-1)--cycle;
\node at (0,0) [vertex_open] (v1) [label=above:{1}] {};
\node at (2,1) [vertex_open] (v2) [label=above:{2}] {};
\node at (2,-1) [vertex_open] (v3) [label=below:{3}] {};
\draw [draw=black, ultra thick] (v1)--(v2);
\draw [draw=black, ultra thick] (v1)--(v3);
\draw [draw=black, ultra thick] (v2)--(v3);
\node at (1,-2) {Illustration of $K_{3}^{\{2,3\}}$} ;
\end{tikzpicture}
\begin{tikzpicture}[vertex_open/.style={circle, draw=black, fill=white}, vertex_closed/.style={circle, draw=black, fill=black}]
\shadedraw[left color=black!15!white, right color=black!70!white, draw=black!1!white] (0,0)--(2,0)--(3.5,2)--cycle;
\shadedraw[left color=black!15!white, right color=black!70!white, draw=black!1!white, opacity=.65] (2,0)--(3.5,2)--(2,4)--cycle;
\shadedraw[left color=black!15!white, right color=black!70!white, draw=black!1!white, opacity=.65] (3.5,2)--(2,4)--(0,4)--cycle;
\shadedraw[left color=black!15!white, right color=black!70!white, draw=black!1!white, opacity=.65] (2,4)--(0,4)--(-1.5,2)--cycle;
\shadedraw[left color=black!15!white, right color=black!70!white, draw=black!1!white, opacity=.65] (0,4)--(-1.5,2)--(0,0)--cycle;
\shadedraw[left color=black!15!white, right color=black!70!white, draw=black!1!white, opacity=.65] (-1.5,2)--(0,0)--(2,0)--cycle;

\node at (0,0) [vertex_open] (v1) [label=below:{1}] {};
\node at (2,0) [vertex_open] (v2) [label=below:{2}] {};
\node at (3.5,2) [vertex_open] (v3) [label=right:{3}] {};
\node at (2,4) [vertex_open] (v4) [label=above:{4}] {};
\node at (0,4) [vertex_open] (v5) [label=above:{5}] {};
\node at (-1.5,2) [vertex_open] (v6) [label=left:{6}] {};
\draw [draw=black, ultra thick] (v1)--(v2)--(v3)--(v4)--(v5)--(v6)--(v1);
\node at (1,-1) {A (tight) cycle $C_{6}^{\{2,3\}}$};
\end{tikzpicture}

\end{center}

Given a family of hypergraphs $\aH$ with common set of edge-types $R$,
we define 
\[\pi_n(\aH):= \max \left\{h_{n}(G)\colon v(G)=n, G\subseteq K^{R}_{n}, \text{ and } G \mbox{ contains no subgraph in }\aH \right\}.\]

A hypergraph $G:=G_n^R$  is {\em extremal} with respect to the family $\aH$ if
\begin{enumerate}
\item $G$ contains no subgraph in $\aH$.
\item $h_n(G)=\pi_n(\aH)$.
\end{enumerate}

The Tur\'an density of $\aH$ is defined to be 
\begin{align*}
	\pi(\mathcal{\aH}):&= \lim_{n\to \infty} \pi_n(\aH) \\
  			   &= \lim_{n\to \infty} \max \left\{\sum_{F\in E(G)} \frac{1}{\binom{n}{|F|}}\colon v(G)=n, G\subseteq K^{R}_{n}, 
			   	\text{ and } G \mbox{ contains no subgraph in }\aH \right\}
\end{align*}
when the limit exists; we will soon show that this limit always exists.
When $\aH$ contains one hypergraph $H$, then we write $\pi(H)$ instead of $\pi(\{H\})$.

Throughout, we will consider $n$ growing to infinity, and $R$ to be a fixed set (not growing with $n$).  
Note that $\pi(\mathcal{H})$ agrees with the \textit{usual} definition of 
\[\displaystyle \pi(\mathcal{H})=\lim_{n\to\infty}\frac{\text{ex}(\mathcal{H},n)}{\binom{n}{k}}\]
when $\mathcal{H}$ is a set of $k$-uniform hypergraphs.
The following result is a direct generalization Katona-Nemetz-Simonovit's theorem \cite{KNS}.

\begin{theorem}
For any family $\aH$ of hypergraphs with a common edge-type $R$,
$\pi(\mathcal{H})$ is well-defined, i.e. the limit $\displaystyle \lim_{n\to\infty}\pi_{n}(\aH)$ exists.
\end{theorem}

\begin{proof}
It suffices to show that $\pi_n(\aH)$, viewed as a sequence in $n$, is decreasing.

Write $R=\{k_1,k_2,...,k_{r}\}$.
Let $G_{n}\subseteq K^{R}_{n}$ be a hypergraph with $v(G_{n})=n$ not containing $\mathcal{H}$ and with Lubell value $h_n(G_{n})=\pi_n(\aH)$.
For any $\ell< n$, consider a random subset $S$ of the vertices of $G$ with size $|S|=\ell$.

Let $G_{n}[S]$ be the induced subgraph of $G_{n}$ (whose vertex set is
restricted to $S$).  Clearly \[\pi_\ell(\aH) \geq
\mathbb{E}(h_{\ell}(G_{n}[S])).\] Write $E(G_{n})=E_{k_1}\bigcup
E_{k_2}\bigcup ...\bigcup E_{k_r}$ where $E_{k_i}$ contains all the
edges of size $k_{i}$.  Note that the expected number of edges of size
$k_i$ in $G_{n}[S]$ is precisely
$\frac{\binom{\ell}{k_i}}{\binom{n}{k_i}}|E_{k_i}|$.  It follows that
\begin{align*}
  \pi_\ell(\aH) &\geq \mathbb{E}(h_{\ell}(G_{n}[S])) \\
  &=\sum_{i=1}^{r} \frac{\mathbb{E}(|E_{k_i} \bigcap \binom{S}{k_i}|)}{\binom{\ell}{k_{i}}} \\
  &=\sum_{i=1}^{r} \frac{ \frac{\binom{\ell}{k_{i}}}{\binom{n}{k_i}}|E_{k_i}|}{\binom{\ell}{k_{i}}} \\
  &=\sum_{i=1}^{r} \frac{ |E_{k_i}| }{\binom{n}{k_i}} \\
  &= h_{n}(G_{n})\\
  &=\pi_n(\aH).
\end{align*}
The sequence $\pi_n(\aH)$ is non-negative and decreasing; therefore it converges.
\end{proof}

For a fixed set $R:=\{k_1,k_2,\ldots, k_r\}$ (with $k_1<k_2<\cdots < k_r$), an {\em $R$-flag} is an $R$-graph containing exactly one edge of each size.
The chain $C^R$ is a special $R$-flag with the edge set
\[\E(C^R)=\{[k_1], [k_2],\ldots, [k_r]\}.\]

\begin{prop}  \label{p1}
For any hypergraph $H$, the following statements hold.
 \begin{enumerate}
\item $|R(H)|-1\leq \pi_n(H)\leq |R(H)|$.
\item For subgraph $H'$ of $H$, we have $\pi(H')\leq \pi(H)- |R(H)|+|R(H')|$.
\item For any $R$-flag $L$ on $m$ vertices and  any $n\geq m$, we have 
$\pi_n(L)=|R|-1$.
  \end{enumerate}
\end{prop}
{\bf Proof:} 
Pick any maximal proper subset $R'$ of $R(H)$.
Consider the complete graph $K_n^{R'}$. Since $K_n^{R'}$ misses one
type of edge in $R(H)\setminus R'$, it does not contain $H$ as a subgraph.
Thus
\[\pi_n(H)\geq h_n(K_n^{R'})=|R'|=|R(H)|-1.\]
The upper bound is due to the fact $h_n(K_n^{R(H)})=|R(H)|$.

Proof of item 2 is similar.  
Let $S=R(H')$ and $G^S_n$ be an extremal hypergraph for $\pi_n(H')$.
Extend $G^S_n$ to $G^{R(H)}_n$ by adding all the edges with cardinalities in $R(H)\setminus S$.
The resulting graph $G^{R(H)}_n$ is $H$-free.
We have
\[\pi_n(H)\geq \pi_n(G^{R(H)}_n)=\pi_{n}(G^S_n)+ |R(H)|-|S|= |R(H)|-|R(H')|+ \pi_n(H').\]
Taking the limit as $n$ goes to infinity, we have
\[\pi(H)\geq |R(H)|-|R(H')|+ \pi(H').\]

Finally, for item 3, consider $L$-free hypergraph $G_n^R$.  Pick a
random $n$-permutation $\sigma$ uniformly.  Let $X$ be the number of
edges of $G^R_n$ hit by a random flag $\sigma(L)$. 
Note that each edge $F$ has probability $\frac{1}{{n\choose |F|}}$ of being hit by $\sigma(L)$.
We have
\begin{equation} \label{eq:ef}
	\E(X)=\sum_{F\in E(G)}\frac{1}{{n\choose |F|}}=h_n(G).  
\end{equation}
Since $G^R_n$ is $L$-free, we have $X\leq r-1$.
Taking the expectation, we have \[h_n(G^R_n)=\E(X)\leq r-1.\]
Hence, $\pi_n(H)\leq r-1$. 
The result is followed after combining with item 1. \hfill $\square$
 
\begin{definition}
A hypergraph $H$ is called {\em degenerate} if $\pi(H)=|R(H)|-1$.
\end{definition}

By Proposition \ref{p1}, flags, and specifically chains, are degenerate hypergraphs.
A necessary condition for $H$ to be degenerate is that every level hypergraph $H^{k_i}$ is $k_i$-partite.
The following examples will show that the converse is not true.

{\bf Example 1:} The complete hypergraph $K^{\{1,2\}}_2$ has three edges
 $\{1\}$, $\{2\}$, and $\{1,2\}$. We claim 
 \begin{equation}
   \label{eq:k12}
   \pi(K^{\{1,2\}}_2)=\frac{5}{4}.
 \end{equation}
The lower bound is from the following construction.
Partition $[n]$ into two parts $A$ and $B$ of nearly equal size.
Consider the hypergraph $G$  with the edge set
$$E(G)={A\choose 1} \cup \left({[n]\choose 2}\setminus {A\choose 2}
\right).$$
It is easy to check $h_n(G)=\frac{5}{4}+O(\frac{1}{n})$ and that $G$ contains no copy of $K_{2}^{\{1,2\}}$.

Now we prove  the upper bound. 
Consider any $K^{\{1,2\}}_2$-free hypergraph $G$ of edge-type $\{1,2\}$ on $n$ vertices.
Let $A$ be the set of all singleton edges. 
For any $x,y\in A$, $xy$ is not a 2-edge of $G$.
We have
\begin{align*}
h_n(G)&\leq \frac{|A|}{n} + 1-\frac{{|A|\choose 2}}{{n\choose 2}} \\
      &= 1+ \frac{|A|}{n} -\frac{|A|^2}{n^2} + O\left(\frac{1}{n}\right) \\
      &\leq 1+ \frac{1}{4} + O\left(\frac{1}{n}\right).
\end{align*}
In the last step, we use the fact that $f(x)=1+x-x^2$ has the maximum value $\frac{5}{4}$. 
Combining the upper and lower bounds we have $\pi(K^{\{1,2\}}_2)=\frac{5}{4}$.

The argument is easily generalized to the complete graph $K^{\{1,k\}}_k$ (for $k>1$).
 We have
\begin{equation}
  \label{eq:k1k}
 \pi(K^{1,k}_k)=1+ \frac{k-1}{k^{\frac{k}{k-1}}}. 
\end{equation}

\begin{definition}
Let $H$ be a hypergraph.
The {\em suspension} of $H$, denoted by $S(H)$, is a new hypergraph with the vertex set $V(H)\cup \{\ast\}$ and
the edge set $\{F\cup \{\ast\}\colon F\in E(H)\}$. 
Here $\ast$ is a new vertex not in $H$.
\end{definition}

\begin{definition}
Let $H$ be a hypergraph.
The $k$-degree of a vertex $x$, denoted $d_{k}(x)$, is the number of edges of size $k$ containing $x$.
\end{definition}

{\bf Example 2:} Consider $H:=S(K^{\{1,2\}}_2)$.
The edges of $H$ are  $\{1,2\}$, $\{2,3\}$, $\{1,2,3\}$.
We claim \[\pi(S(K^{\{1,2\}}_2))= \frac{5}{4}.\]

The lower bound is from the following construction.
Partition $[n]$ into two parts $A$ and $B$ of nearly equal size.
Consider the hypergraph $G$ with the edge set $E=E_{2}\bigcup E_{3}$ where $E_{2}=\binom{A}{2}\bigcup \binom{B}{2}$ and
$E_{3}=\binom{[n]}{3}\setminus \left(\binom{A}{3} \bigcup \binom{B}{3}\right)$. 
It is easy to check $h_n(G)=\frac{5}{4}+O\left(\frac{1}{n}\right)$ and that $G$ is $H$-free.

Now we prove  the upper bound. 
Consider any $H$-free hypergraph $G$ of edge-type $\{2,3\}$ on $n$ vertices.
Recall that $d_{2}(v)$ denotes the number of 2-edges that contain $v$.
For each pair of 2-edges that intersect $v$ there is a unique 3-set containing those two pairs.
This 3-set cannot appear in the edge set of $G$ since $G$ is $H$-free.
We say that the edge is forbidden.
Note that each 3-edge may be forbidden up to 3 times in this manner--depending on which of the three vertices we call $v$.
Hence there are at least $\frac{1}{3} \sum_{v\in [n]} \binom{d_{2}(v)}{2}$ 3-edges which are not in $G$.
Note that this is true for any $H$-free graph $G$ with number of vertices.
Hence
\begin{align*}
h_n(G) &\leq \frac{ \binom{n}{3}-\frac{1}{3}\sum_{v\in [n]} \binom{d_{2}(v)}{2} }{\binom{n}{3}} + \frac{ \frac{1}{2}\sum_{v\in [n]} d_{2}(v)}{\binom{n}{2}} \\                
&= 1-\frac{\sum_{v\in [n]} d_2(v)^{2}}{6\binom{n}{3}} +\left(\frac{1}{6\binom{n}{3}}+\frac{1}{2\binom{n}{2}}\right)\sum_{v\in [n]} d_2(v) +1.
\end{align*}
Applying Cauchy-Schwarz Inequality and letting $m:=\sum_vd_2(v)$,
 we have
\begin{align*}
h_n(G) &\leq \frac{-\left(\sum_{v\in [n]} d_2(v)\right)^2}{6n\binom{n}{3}}
+\left(\frac{1}{6\binom{n}{3}}+\frac{1}{2\binom{n}{2}}\right)\sum_{v\in [n]} d_2(v) + 1 \\
&=-\frac{m^2}{n^4} + \frac{m}{n^2}+1 +O\left(\frac{1}{n}\right)\\
&\leq \frac{5}{4}  +O\left(\frac{1}{n}\right).
\end{align*}
In the last step, we use the fact that $f(x)=1+x-x^2$ has the maximum value $\frac{5}{4}$.
Taking the limit, we get $\pi(S(K_{2}^{\{1,2\}}))\leq \frac{5}{4}.$

We can generalize this construction, giving the following lower bound
for $S^k(K_2^{\{1,2\}})$ (the $k$-th suspension of $K_2^{\{1,2\}}$).
The details of the computation are omitted.
\begin{equation} \label{eq:sk122}
	\pi(S^k(K_2^{\{1,2\}}))\geq 1+ \frac{1}{2^{k+1}}.
\end{equation}

\begin{conjecture}\label{conj:1}
  For any $k\geq 2$, 
$\pi(S^k(K_2^{\{1,2\}}))= 1+ \frac{1}{2^{k+1}}.$
\end{conjecture}

\section{Supersaturation and Blowing-up}

Supersaturation Lemma \cite{ErSi} is an important tool for uniform hypergraphs.
There is a natural generalization of the supersaturation lemma
and blowing-up in non-uniform hypergraphs.

\begin{lemma} {\bf (Supersaturation)} For any hypergraph $H$
and $a>0$ there are $b$, $n_0>0$ so that if $G$ is a hypergraph on $n>n_0$
vertices with $R(G)=R(H)$ and $h_n(G)>\pi(H)+a$ then
$G$ contains at least $b{n\choose v(H)}$ copies of $H$.  
\end{lemma}

{\bf Proof:}  Let $R:=R(H)$ and $r:=|R|$.
Since $\displaystyle \pi(H)=\lim_{n\to\infty}\pi_n(H)$, there is an $n_0$ so that for $m\geq n_0$, $\pi_m(H)\leq \pi(H)+\frac{a}{2}$.
For any $n_0\leq m\leq n$, there must be at least $\frac{a}{2r}{n\choose m}$ $m$-sets $M\subset V(G)$ inducing 
a $R$-graph $G[M]$ with $h(G[M])>\pi(H)+\frac{a}{2}$.
Otherwise, we have 
\[\sum_{M}h_m(G[M])\leq \left(\pi(H)+\frac{a}{2}\right){n\choose m} + \frac{a}{2r}{n\choose m}r=(\pi(H)+a){n\choose m}.\]
But we also have
\begin{align*}
\sum_Mh_m(G[M]) &=\sum_M \sum_{\stackrel{F\in E(G)}{F\subseteq M}}\frac{1}{{m\choose |F|}} \\
                &=\sum_{F\in E(G)}\sum_{M\supseteq F}\frac{1}{{m\choose |F|}}  \\
                &= \sum_{F\in E(G)} \frac{{n-|F| \choose m-|F|}}{{m\choose |F|}}  \\
                &=\sum_{F\in E(G)} \frac{{n\choose m}}{{n\choose |F|}}\\
                &= {n\choose m}h_n(G).
\end{align*}
This is a contradiction to the assumption that $h_n(G)>\pi(H)+a$.
By the choice of $m$, each of these $m$-sets contains a copy of $H$, so the number of copies of $H$ in $G$ is at least 
$\frac{a}{2r}{n\choose m}/{{n-v(H)\choose m-v(H)}}=b {n\choose v(H)}$ where $b:=\frac{a}{2r}{m\choose v(H)}^{-1}$. \hfill $\square$

Supersaturation can be used to show that ``blowing up'' does not change the Tur\'an density $\pi(H)$ just like in the uniform cases.

\begin{definition}
For any hypergraph $H_n$ and  positive integers $s_1,s_2,\ldots, s_n$, the {\em blowup} of $H$ is a new hypergraph $(V,E)$, denoted by 
$H_n(s_1,s_2,\ldots, s_n)$, satisfying
\begin{enumerate}
\item $V:=\sqcup_{i=1}^n V_i$, where $|V_i|=s_i$.
\item $E=\cup_{F\in \E(H)} \prod_{i\in F} V_i$.
\end{enumerate}
When $s_1=s_2=\cdots=s_n=s$, we simply write it as $H(s)$.
\end{definition}

Consider the following simple example.
Take $H$ to be the hypergraph with vertex set $[3]$ and edge set $E=\{\{1,2\}, \{1,2,3\}\}$, a chain.
Consider the blow-ups $H(2,1,1)$ and $H(1,1,2)$ illustrated below.

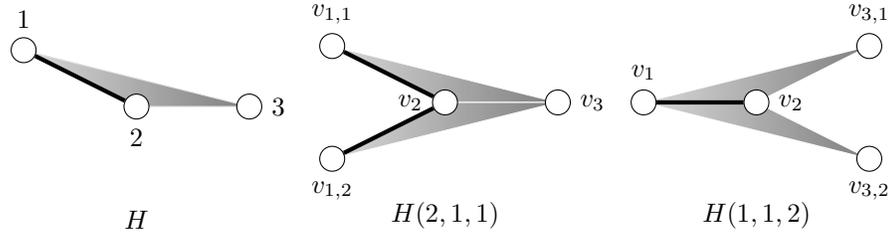
\begin{center}
\begin{tikzpicture}[vertex/.style={circle, draw=black, fill=white}, scale=.75]
\shadedraw[left color=black!15!white, right color=black!50!white, draw=black!10!white] (0,1)--(2,0)--(4,0)--cycle;
\node at (0,1) [vertex] (v1) [label=above:{1}] {};
\node at (2,0) [vertex] (v2) [label=below:{2}] {};
\node at (4,0) [vertex] (v3) [label=right:{3}] {};
\draw [draw=black, ultra thick] (v1)--(v2);
\node at (2,-2) (empty) {$H$};
\end{tikzpicture}
\begin{tikzpicture}[vertex/.style={circle, draw=black, fill=white}, scale=.75]
\shadedraw[left color=black!15!white, right color=black!50!white, draw=white] (0,1)--(2,0)--(4,0)--cycle;
\shadedraw[left color=black!15!white, right color=black!50!white, draw=white] (0,-1)--(2,0)--(4,0)--cycle;
\node at (0,1) [vertex] (v11) [label=above:{$v_{1,1}$}] {};
\node at (0,-1) [vertex] (v12) [label=below:{$v_{1,2}$}] {};
\node at (2,0) [vertex] (v2) [label=left:{$v_2$}] {};
\node at (4,0) [vertex] (v3) [label=right:{$v_3$}] {};
\draw [draw=black, ultra thick] (v11)--(v2);
\draw [draw=black, ultra thick] (v12)--(v2);
\node at (2,-2) {$H(2,1,1)$};
\end{tikzpicture}
\begin{tikzpicture}[vertex/.style={circle, draw=black, fill=white}, scale=.75]
\shadedraw[left color=black!15!white, right color=black!50!white, draw=white] (0,0)--(2,0)--(4,1)--cycle;
\shadedraw[left color=black!15!white, right color=black!50!white, draw=white] (0,0)--(2,0)--(4,-1)--cycle;
\node at (0,0) [vertex] (v1) [label=above:{$v_{1}$}] {};
\node at (2,0) [vertex] (v2) [label=right:{$v_{2}$}] {};
\node at (4,1) [vertex] (v31) [label=above:{$v_{3,1}$}] {};
\node at (4,-1) [vertex] (v32) [label=below:{$v_{3,2}$}] {};
\draw [draw=black, ultra thick] (v1)--(v2);
\node at (2,-2) {$H(1,1,2)$};
\end{tikzpicture}
\end{center}

\noindent In the blow-up $H(2,1,1)$ vertex 1 splits into vertices $v_{1,1}$ and $v_{1,2}$; vertex 2 becomes $v_2$ and vertex 3 becomes $v_3$.
In the blow-up $H(1,1,2)$ vertex 3 splits into vertices $v_{3,1}$ and $v_{3,2}$; vertex 1 becomes $v_1$ and vertex 2 becomes $v_2$.

\begin{theorem} {\bf (Blowing up)} \label{blowup}
Let $H$ be any hypergraph and let $s\geq 2$.
Then $\pi(H(s))=\pi(H)$.
\end{theorem}

{\bf Proof:} 
Let $R:=R(H)$.
By the supersaturation lemma, for any $a>0$ there is a $b>0$ and an $n_0$ so that any $R$-graph $G$ on $n\geq n_0$ vertices with $h_n(G)>\pi(H)+a$
contains at least $b{n \choose v(H)}$ copies of $H$.  
Consider an auxiliary $v(H)$-uniform hypergraph $U$ on the same vertex set as $G$ where edges of $U$ correspond to copies of $H$ in $G$.
For any $S>0$, there is a copy of $K=K^{v(H)}_{v(H)}(S)$ in $U$.
This follows from the fact that $\pi(K^{v(H)}_{v(H)}(S))=0$ since it is $v(H)$-partite, and $h_{n}(U)=b>0$.
Now color each edge of $K$ by one of $v(H)!$ colors, each color corresponding to one of $v(H)!$ possible orders the vertices of $H$ are mapped to the parts of $K$.
The pigeon-hole principle gives us that one of the color classes contains at least $S^{v(H)}/v(H)!$ edges.
For large enough $S$ there is a monochromatic copy of $K^{v(H)}_{v(H)}(s)$, which gives a copy of $H(s)$ in $G$. \hfill $\square$

\begin{corollary}{\bf (Squeeze Theorem)}
Let $H$ be any hypergraph.
If there exists a hypergraph $H^{\prime}$ and integer $s\geq 2$ such that $H^{\prime}\subseteq H\subseteq H^{\prime}(s)$ then $\pi(H)=\pi(H^{\prime})$.
\end{corollary}

{\bf Proof:} 
One needs only observe that for any hypergraphs $H_{1}\subseteq H_{2}\subseteq H_{3}$ that $\pi(H_{1})\leq \pi(H_{2})\leq \pi(H_{3})$.
If $H_{3}=H_{1}(s)$ for some $s\geq 2$ then $\pi(H_{1})=\pi(H_{3})$ by the previous theorem. \hfill $\square$

\section{Tur\'an Densities of $\{1,2\}$-hypergraphs}

In this section we will determine the Tur\'an density for any hypergraph $H$ with $R(H)=\{1,2\}$.
We begin with the following more general result.

\begin{theorem}
Let $H=H^{1}\cup H^{k}$ be a hypergraph with $R(H)=\{1,k\}$ and $E(H^{1})=V(H^{k})$.
Then \[\pi(H) = \begin{cases} 1+\pi(H^{k}) & \text{if } \pi(H^{k})\geq 1-\frac{1}{k}; \\
                              1+\left(\frac{1}{k(1-\pi(H^{k}))}\right)^{1/(k-1)}\left(1-\frac{1}{k}\right) & \text{otherwise.} \end{cases}\]
\end{theorem}

{\bf Proof:}
For each $n\in \mathbb{N}$, let $G_{n}$ be any $H$-free graph $n$ vertices with $h_{n}(G_{n})=\pi_{n}(H)$.
Partition the vertices of $G_{n}$ into $X_{n}=\{v\in V(G_{n}):\{v\}\in E(G)\}$ and $\bar{X_{n}}$ containing everything else.
Say that $|X_{n}|=x_{n}n$ and $|\bar{X_{n}}|=(1-x_{n})n$ for some $x_{n}\in [0,1]$.
Since $(x_{n})$ is a sequence in $[0,1]$ it has a convergent subsequence.
Consider $(x_{n})$ to be the convergent subsequence, and say that $x_{n}\to x\in [0,1]$.
With the benefit of hindsight, we know that $x>0$, however, for the upper bound portion of this proof we will not assume this knowledge.

Since there is no copy of $H$ in $G_{n}$, it follows that $G_{n}[X_{n}]$ contains no copy of $H^{k}$.
We have that
\begin{align*}
\pi(H) &=\lim_{n\to\infty} h_{n}(G_{n}) \\
             &=\lim_{n\to\infty} \sum_{F\in H^{1}}\frac{1}{\binom{n}{1}} + \sum_{F\in H^{k}} \frac{1}{\binom{n}{k}} \\
             &\leq \lim_{n\to\infty} \frac{x_{n}n}{\binom{n}{1}} + \frac{\binom{n}{k}-(1-\pi_{x_{n}n}(H^{k}))\binom{x_{n}n}{k}}{\binom{n}{k}} \\
	     &=\lim_{n\to\infty} 1 + x_{n} - (1-\pi_{x_{n}n}(H^{k}))\frac{\binom{x_{n}n}{k}}{\binom{n}{k}} \\
	     &\leq \lim_{n\to\infty} \begin{cases} 1+\frac{1}{\sqrt{n}} & \text{if } x_{n}n \leq \sqrt{n}, \\
	                                            1+ x_{n} - (1-\pi_{x_{n}n}(H^{k}))x_{n}^{k} & \text{if } x_{n}n>\sqrt{n}, \end{cases} \\
	     &\leq \max\{1, 1+x-(1-\pi(H^{k}))x^{k}\}.
\end{align*}
Let $f(x)=1+x-(1-\pi(H^{k}))x^{k}$ and then note that \[\pi(H)= \lim_{n\to\infty} h_{n}(G) \leq \max_{x\in [0,1]} f(x).\]
An easy calculus exercise shows that $f^{\prime\prime}(x)<0$ for all
$x>0$, and $f^{\prime}(x)=0$ when 
$x=\left(\frac{1}{k(1-\pi(H^{k}))}\right)^{\frac{1}{k-1}}.$
If $\frac{1}{k(1-\pi(H^{k}))}\geq 1$ then $f^{\prime}(x)>0$ when $x\in [0,1)$ and hence $f(x)$ is maximized when $x=1$.
Note that $f(1)=1+\pi(H^{k})$.
If, on the other hand, $\frac{1}{k(1-\pi(H^{k}))}< 1$ it follows that $f(x)$ is maximized at $x=\left(\frac{1}{k(1-\pi(H^{k}))}\right)^{1/(k-1)}$.
Together, this gives us
\[\pi(H) \leq \begin{cases} 1+\pi(H^{k}) & \text{if } \pi(H^{k})\geq 1-\frac{1}{k}; \\
                            1+\left(\frac{1}{k(1-\pi(H^{k}))}\right)^{1/(k-1)}\left(1-\frac{1}{k}\right) & \text{otherwise}. \end{cases}\]
To get equality, take $x$ that maximizes $f(x)$ as above.
For any $n\in \mathbb{N}$ (thinking of $n\to \infty$) partition $[n]$ into two sets $X$ and $\bar{X}$ with $|X|=xn$ and $|\bar{X}|=(1-x)n$.
Let $E(G^{1})=\{\{v\}:v\in X\}$ and let $g^{k}$ be a $k$-uniform graph on $xn$ vertices attaining $|E(g^{k})|=\text{ex}(xn,H^{k})$ and $g^{k}$ is $H^{k}$-free.
Then \[E(G^{k})=\{F\in \binom{[n]}{k}:\text{either } F\in E(g^{k}) \text{ or } F\cap \bar{X}\neq \emptyset\}.\]
Then $G=G^{1}\cup G^{k}$ is $H$-free and (by choice of $x$) we have that $\displaystyle \lim_{n\to\infty}h_{n}(G)$ attains the upper bound of $\pi(H)$.  \hfill $\square$

Let us now return to the task of determining $\pi(H)$ when $H=H^{1}\cup H^{2}$.

\begin{prop}
Let $H=H^{1}\cup H^{2}$.
If $H^{2}$ is not bipartite, then \[\pi(H)=1+\pi(H^{2})= 1+\left(1-\frac{1}{\chi(H^{2})-1}\right)=2-\frac{1}{\chi(H^{2})-1}.\]
\end{prop}

{\bf Proof:}
First, $\pi(H)\geq 1+\pi(H^{2})$ since one can construct an $H$-free graph $G_{n}$ by letting
\[E(G_{n})=\{\{v\}:v\in V(G_{n})\}\cup E(G^{\prime}_{n})\]
where $G^{\prime}_{n}$ attains $h_{n}(G^{\prime}_{n})=\pi_{n}(H^{2})$ and $G^{\prime}_{n}$ is $H^{2}$-free.
Then 
\[\pi(H)\geq \lim_{n\to\infty} h_{n}(G_{n}) = \lim_{n\to\infty} 1+\pi_{n}(H^{2}) = 1+\pi(H^{2}).\]

To get the upper-bound, first add every missing $1$-edge into $H$, call the new graph $H^{\prime}$.  
Note that $\pi(H)\leq \pi(H^{\prime})$.
Note that we didn't change the edge set $H^{2}$.
The Erd\H{o}s-Stone-Simonovits theorem states that if $H^{2}$ is not bipartite, then $\pi(H^{2})=1-\frac{1}{\chi(H^{2})-1}$.
Also, if $H^{2}$ is not bipartite, then $\chi(H^{2})\geq 3$.
With the added vertices, taking $k=2$, we apply the previous theorem.
Since 
\[\pi(H^{2})=1-\frac{1}{\chi(H^{2})-1}\geq 1-\frac{1}{2}\] we may conclude that $\pi(H)\leq \pi(H^{\prime})=1+\pi(H^{2})$. \hfill $\square$

It remains to investigate the cases when $H^{2}$ is bipartite.

\begin{prop}
Let $H=H^{1}\cup H^{2}$.
If $H^{2}$ is bipartite and $K_{2}^{\{1,2\}}\subseteq H$ then $\pi(H)=\frac{5}{4}$.
\end{prop}

{\bf Proof:}
First, in example 1, we computed $\pi(K_{2}^{\{1,2\}})=\frac{5}{4}$.
Second, $H$ must be contained in some blow-up of $K_{2}^{\{1,2\}}$ since $H^{2}$ is bipartite, i.e. there exists some $s>2$ such that $H\subseteq K_{2}^{\{1,2\}}(s)$.
So, by the squeeze theorem we have 
\[\frac{5}{4}=\pi(K_{2}^{\{1,2\}}) \leq \pi(H) \leq \pi(K_{2}^{\{1,2\}}(s))=\frac{5}{4}.\]
Hence $\pi(H)=\frac{5}{4}$ as claimed.
\hfill $\square$

\begin{definition}
We will say that $H=H^{1}\cup H^{2}$ is a \textbf{\textit{closed path}} (from $x_{1}$ to $x_{k}$) of length $k$
if $V(H)=\{x_{1}, x_{2},...,x_{k}\}$ and $E(H^{1})=\{\{x_{1}\}, \{x_{k}\}\}$ and $E(H^{2})=\{ \{x_{i}, x_{i+1}\}:1\leq i\leq k-1\}$.
We will denote a closed path of length $k$, or a closed $k$-path, by $\bar{P}_{k}$.
\end{definition}

Pictorially, we view a closed path of length $k$ as follows:

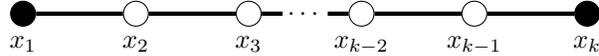
\begin{center}
\begin{tikzpicture}[vertex_closed/.style={circle, draw=black, fill=black}, vertex_open/.style={circle, draw=black, fill=white}, scale=1.5]
\node at (0,0) [vertex_closed] (v1) [label=below:{$x_{1}$}] {};
\node at (1,0) [vertex_open] (v2) [label=below:{$x_{2}$}] {};
\node at (2,0) [vertex_open] (v3) [label=below:{$x_{3}$}] {};
\node at (3,0) [vertex_open] (v4) [label=below:{$x_{k-2}$}] {};
\node at (4,0) [vertex_open] (v5) [label=below:{$x_{k-1}$}] {};
\node at (5,0) [vertex_closed] (v6) [label=below:{$x_{k}$}] {};
\draw [draw=black, ultra thick] (v1)--(v2);
\draw [draw=black, ultra thick] (v2)--(v3);
\draw [draw=black, ultra thick] (v4)--(v5);
\draw [draw=black, ultra thick] (v5)--(v6);
\draw [draw=black, ultra thick] (v3)--(2.3, 0);
\draw [draw=black, ultra thick] (2.7, 0)--(v4);
\node at (2.5, 0) {$\dots$};
\end{tikzpicture}
\end{center}

\begin{prop}
Let $H=H^{1}\cup H^{2}$.
If $H^{2}$ is bipartite and $H$ does not contain a copy of $K_{2}^{\{1,2\}}$ and $H$ contains a closed path of length $2k$, then $\pi(H)=\frac{9}{8}$.
\end{prop}

{\bf Proof:}
First, we will give a construction giving us the lower bound.
For any $n\in \mathbb{N}$ let $G_{n}$ have vertex set $[n]$.
Partition the vertices of $G_{n}$ into two sets $X$ and $\bar{X}$ where $|X|=\frac{3n}{4}$ and $|\bar{X}|=\frac{n}{4}$.
Let \[E(G)=\{\{x\}:x\in X\} \cup \{\{x,\bar{x}\}: x\in X \text{ and } \bar{x}\in \bar{X}\}.\]
It is clear that $G_{n}$ contains no closed paths of length $2k$ when $k\geq 1$.
Also,
\begin{align*}
\lim_{n\to\infty} h_{n}(G_{n}) &= \lim_{n\to\infty} \frac{|X|}{\binom{n}{1}}+\frac{|X|\cdot |\bar{X}|}{\binom{n}{2}} \\
             &= \lim_{n\to\infty} \frac{3}{4}+\frac{\frac{3}{16} n^{2}}{\binom{n}{2}} \\
	     &= \frac{3}{4}+\frac{3}{8} =\frac{9}{8}.
\end{align*}
Thus $\pi(H)\geq \frac{9}{8}$ for any $H$ containing a closed path of length $2k$ for any $k\geq 1$.

Since $H^{2}$ is bipartite, and $H^{2}$ does not contain a copy of $K_{2}^{\{1,2\}}$, then $H$ is contained in a blow-up of a closed $4$-path.
To see this, note that there is a bipartition of the vertices of $H$, $V(H)=A\cup B$, (with respect to the $2$-edges in $H$).
Furthermore, we can partition $A$ into $A_{1}\cup A_{2}$ where $v\in A_{1}$ if $\{v\}\in E(H)$ and $v\in A$, $v\in A_{2}$ if $v\in A\setminus A_{1}$.
And similarly partition $B$ into $B_{1}\cup B_{2}$ with $v\in B_{1}$ if $\{v\}\in E(H)$ and $v\in B$.
Then note that there are no edges from $A_{1}$ to $B_{1}$ since $H$ contains no copy of $K_{2}^{\{1,2\}}$.
So $H\subset \bar{P}_{4}(\max\{|A_{1}|, |A_{2}|, |B_{1}|, |B_{2}|)$--a blow-up of $\bar{P}_{4}$.
Below is a graphical representation of $H$, illustrating that $H$ is contained in a blow-up of $\bar{P}_{4}$.
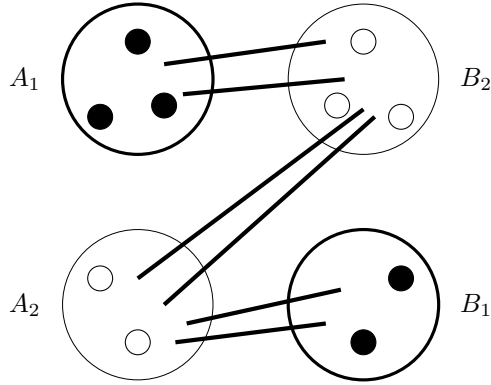
\begin{center}
\begin{tikzpicture}[vertex_closed/.style={circle, draw=black, fill=black}, vertex_open/.style={circle, draw=black, fill=white}]
\draw [draw=black, very thick] (0, 0) circle (1 cm);
\draw [draw=black, very thin] (0,-3) circle (1 cm);
\draw [draw=black, very thin] (3, 0) circle (1 cm);
\draw [draw=black, very thick] (3,-3) circle (1 cm);
\node at (-1.5, 0) {$A_{1}$};
\node at (-1.5,-3) {$A_{2}$};
\node at (4.5, 0) {$B_{2}$};
\node at (4.5, -3) {$B_{1}$};

\node at (0, .5) [vertex_closed] (v1) {};
\node at (-.5, -.5) [vertex_closed] (v2) {};
\node at (.35,-.35) [vertex_closed] (v3) {};

\node at (3, .5) [vertex_open] (v4) {};
\node at (2.65,-.35) [vertex_open] (v5) {};
\node at (3.5, -.5) [vertex_open] (v6) {};

\node at (0, -3.5) [vertex_open] (v7) {};
\node at (-.5, -2.65) [vertex_open] (v8) {};

\node at (3, -3.5) [vertex_closed] (v9) {};
\node at (3.5, -2.65) [vertex_closed] (v10) {};

\draw [draw=black, ultra thick] (.6,-.2)--(2.75, 0);
\draw [draw=black, ultra thick] (.35, .2)--(2.5, .5);
\draw [draw=black, ultra thick] (3,-.4)--(0,-2.65);
\draw [draw=black, ultra thick] (3.15, -.5)--(.35, -3);
\draw [draw=black, ultra thick] (.5, -3.5)--(2.5, -3.25);
\draw [draw=black, ultra thick] (.65,-3.25)--(2.7, -2.8);
\end{tikzpicture}
\end{center}

Since $\pi(H)\leq \pi(\bar{P}_{4}(s))=\pi(\bar{P}_{4})$ we need only show that $\pi(\bar{P}_{4})\leq \frac{9}{8}$.
Let $G_{n}$ be a family of $\bar{P}_{4}$-free graphs such that $h_{n}(G_{n})=\pi_{n}(\bar{P}_{4})$.
Partition the vertices of $G_{n}$ as follows:
\begin{align*}
X_{n} &= \{v:\{v\}\in E(G_{n})\}, \\
Y_{n} &= \{v: \{v\}\notin E(G_{n}) \text{ and } \exists x_{1}\neq x_{2}\in X_{n} \text{ with } \{x_{1}, v\}, \{x_{2}, v\} \in E(G_{n})\}, \\
Z_{n} &= V(G)\setminus (X_{n}\cup Y_{n}).
\end{align*}
Let us say that $|X_{n}|=xn$, $|Y_{n}|=yn$ and hence $|Z_{n}|=(1-x-y)n$.

First, note that $E(G)\cap \binom{Y_{n}}{2}=\emptyset$.
Otherwise, since each vertex in $Y_{n}$ has at least 2 neighbors in $X_{n}$, $G_{n}$ would contain a closed path of length $4$.
Also, each vertex in $Z_{n}$ has at most 1 neighbor in $X_{n}$.
It follows that
\begin{align*}
\pi(\bar{P}_{4}) &=\lim_{n\to\infty} \pi_{n}(\bar{P}_{4}) \\
           &= \lim_{n\to \infty} h_{n}(G_{n}) \\
           &\leq \lim_{n\to\infty} \frac{|X_{n}|}{\binom{n}{1}} + \frac{|X_{n}|\cdot |Y_{n}|}{\binom{n}{2}} + \frac{|Y_{n}|\cdot |Z_{n}|}{\binom{n}{2}} + \frac{\binom{|Z_{n}|}{2}}{\binom{n}{2}} + \frac{|Z_{n}|}{\binom{n}{2}} \\
	   &\leq \lim_{n\to\infty} \frac{xn}{\binom{n}{1}} + \frac{xyn^{2}}{\binom{n}{2}} + \frac{y(1-x-y)n^{2}}{\binom{n}{2}} + \frac{\frac{(1-x-y)^{2}n^{2}}{2}}{\binom{n}{2}} + \frac{(1-x-y)n}{\binom{n}{2}} \\
	   &\leq \max_{\substack {0\leq x\leq 1 \\ 0\leq y\leq 1-x}} x + 2xy + 2y(1-x-y)+ (1-x-y)^{2} \\
	   &=\frac{9}{8}.
\end{align*}
The last inequality is an easy multivariate calculus exercise.
One can also verify it with software, such as \textit{Mathematica}, the syntax being:
\begin{center}
\begin{lstlisting}
Maximize[{x^2-x-y^2+2*x*y+1, 0<=x<=1, 0<=y<=1-x}, {x,y}].
\end{lstlisting}
\end{center}
It may be of interest to note that the maximum value of the function is obtained when $x=\frac{3}{4}$ and $y=\frac{1}{4}$.
In this case $Z_{n}$ is empty.
Since our upper bound matches the lower bound, we have the desired result.  \hfill $\square$

\begin{prop}
Let $H=H^{1}\cup H^{2}$.
If $H^{2}$ is bipartite and $H^{2}$ does not contain a closed $2k$-path for any $k\geq 1$, then $\pi(H)=1$.
\end{prop}

{\bf Proof:}
First, since $|R(H)|=2$ we have, trivially, that $\pi(H)\geq 1$.
Since $H$ contains no path of length $2k$ for any $k\geq 1$ it must be the case that $H$ is contained in a blow-up of a chain $C^{\{1,2\}}=\{\{x\}, \{x,y\}\}$.
This is most clearly seen by again, considering the previous illustration.  The difference is, in this case, $B_{1}$ (or $A_{1}$) is empty.
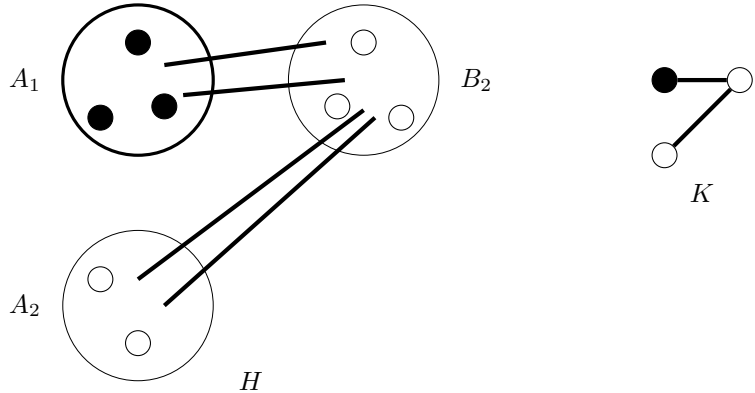
\begin{center}
\begin{tikzpicture}[vertex_closed/.style={circle, draw=black, fill=black}, vertex_open/.style={circle, draw=black, fill=white}]
\draw [draw=black, very thick] (0, 0) circle (1 cm);
\draw [draw=black, very thin] (0,-3) circle (1 cm);
\draw [draw=black, very thin] (3, 0) circle (1 cm);
\node at (-1.5, 0) {$A_{1}$};
\node at (-1.5,-3) {$A_{2}$};
\node at (4.5, 0) {$B_{2}$};

\node at (0, .5) [vertex_closed] (v1) {};
\node at (-.5, -.5) [vertex_closed] (v2) {};
\node at (.35,-.35) [vertex_closed] (v3) {};

\node at (3, .5) [vertex_open] (v4) {};
\node at (2.65,-.35) [vertex_open] (v5) {};
\node at (3.5, -.5) [vertex_open] (v6) {};

\node at (0, -3.5) [vertex_open] (v7) {};
\node at (-.5, -2.65) [vertex_open] (v8) {};

\draw [draw=black, ultra thick] (.6,-.2)--(2.75, 0);
\draw [draw=black, ultra thick] (.35, .2)--(2.5, .5);
\draw [draw=black, ultra thick] (3,-.4)--(0,-2.65);
\draw [draw=black, ultra thick] (3.15, -.5)--(.35, -3);

\node at (1.5, -4) {$H$};

\node at (7,0) [vertex_closed] (w1) {};
\node at (8,0) [vertex_open] (w2) {};
\node at (7,-1) [vertex_open] (w3) {};
\draw [draw=black, ultra thick] (w1)--(w2)--(w3);
\node at (7.5, -1.5) {$K$};
\end{tikzpicture}
\end{center}
It is clear that $H$ is contained in a blow-up of $K$ where \[K=\{\{x\}, \{x,y\}, \{y,z\}\}\subseteq C^{\{1,2\}}(2,1)=\{\{x\}, \{z\}, \{x, y\}, \{z, y\}\}.\]
It follows that $\pi(H)\leq \pi(C^{\{1,2\}})=1$. \hfill $\square$

The combination of these propositions completely determines $\pi(H)$ when $R(H)=\{1,2\}$.
The results are summarized by the following theorem.

\begin{theorem}\label{t12}
For any hypergraph $H$ with  $R(H)=\{1,2\}$, we have 
\[\pi(H) = \begin{cases} 2-\frac{1}{\chi(H^{2})-1} & \text{if } H^{2} \text{ is not bipartite}; \\
                         \frac{5}{4}             & \text{if } H^{2} \text{ is bipartite and } \min \{k:\bar{P}_{2k}\subseteq H\}=1; \\
			 \frac{9}{8}             & \text{if } H^{2} \text{ is bipartite and } \min \{k:\bar{P}_{2k}\subseteq H\}\geq 2; \\
			 1			 & \text{if } H^{2} \text{ is bipartite and } \bar{P}_{2k}\nsubseteq H \text{ for any }k\geq 1. \end{cases}\]
\end{theorem}

\section{Degenerate hypergraphs}

Recall that a hypergraph $H$ is {\em degenerate} if $\pi(H)=|R(H)|-1$.
For $k$-uniform hypergraph $H$, $H$ is degenerate if and only $H$ is $k$-partite.
From Proposition \ref{p1} and Theorem \ref{blowup}, we have the following proposition.
\begin{prop}
Suppose $H$ is a degenerate hypergraph. 
Then the following properties hold.
\begin{itemize}
\item Every subgraph of $H$ is degenerate.
\item Every blowup of $H$ is degenerate.
\item Any subgraph of the blowup of a flag is degenerate. 
\end{itemize}
\end{prop}

Note that every flag is a subgraph of some blowup of a chain with
the same edge type.
Is every degenerate hypergraph a subgraph of some blowup of a chain?
The answer is yes for uniform hypergraphs and $\{1,2\}$-hypergraphs.
This follows from Theorem \ref{t12}, which completely determined $\pi(H)$ when $R(H)=\{1,2\}$,
and from the fact that a $k$-uniform hypergraph is degenerate if and only if it is $k$-partite (a subgraph of a blowup of a single edge).
However, the answer in general is false.
We will show that the following hypergraph $H_1$ with edge set $E(H_1)=\{\{1,2\}, \{1,3\}, \{2,3,4,\}\}$ is degenerate. 

\begin{center}
\begin{tikzpicture}[wvertex/.style={circle, draw=black, fill=white}, bvertex/.style={circle, draw=black, fill=black}] 
\draw[draw=black!50!white, fill=black!50!white] (2,1)--(4,0)--(2,-1)--cycle;
\node at (0,0) [wvertex] (v1) [label=above:{1}] {};
\node at (2,1) [wvertex] (v2) [label=above:{2}] {};
\node at (2,-1) [wvertex] (v3) [label=below:{3}] {};
\node at (4,0) [wvertex] (v4) [label=above:{4}] {};
\draw [draw=black, ultra thick] (v1)--(v2);
\draw [draw=black, ultra thick] (v1)--(v3);
\end{tikzpicture} \\
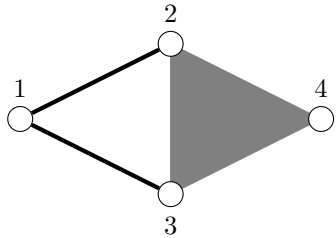
$H_{1}$: A degenerate hypergraph not contained in the blowup of a chain.
\end{center}

This result is a special case of the following theorem.

\begin{definition}
Let $H$ be a hypergraph containing some $2$-edges.
The $2$-subdivision of $H$ is a new hypergraph $H'$ obtained from $H$ by subdividing each $2$-edge simultaneously.
Namely, if $H$ contains $t$ $2$-edges, add $t$ new vertices $x_1,x_2,\ldots,x_t$ to $H$ and for $i=1,2,\ldots, t$ replace the $2$-edge $\{u_{i},v_{i}\}$ with $\{u_i,x_i\}$ and $\{x_i,v_i\}$.
\end{definition}

\begin{theorem}\label{t:sd}
Let $H'$ be the $2$-subdivision of $H$.
If $H$ is degenerate, then so is $H'$.
\end{theorem}

For example, $H_1$ can be viewed as the $2$-division of the chain $C^{\{2,3\}}$.
Since any chain is degenerate, so is $H_1$.
To prove this theorem, we need a Lemma on graphs, which has independent interest.

\begin{definition}
Let $G$ be any simple graph.
Then $G^{(2)}$, a variation of the square of $G$, will be defined as follows:
\begin{itemize}
\item $V(G^{(2)}):=V(G)$,
\item $E(G^{(2)}):=\{ \{u,v\}| \exists w\in V(G) \text{ with } \{u,w\},\{v,w\}\in E(G)\}$.
\end{itemize}
\end{definition}

Note that an edge of $G$ may or may not be an edge of $G^{(2)}$.
For example, if $G$ is the complete graph, then $G^{(2)}$ is also the complete graph.
However, if $G$ is a complete bipartite graph with partite set $V_1\cup V_2$,
then $G^{(2)}$ is the disjoint union of two complete graphs on $V_1$ and $V_2$.
In this case, $G^{(2)}$ is the complement graph of $G$!
We also note that $G^{(2)}$ is the empty graph if $G$ is a matching.
Surprisingly, we have the following Lemma on the difference of
the number of edges in $G$ and $G^{(2)}$.

\begin{lemma}\label{l:square}
For any simple graph $G$ on $n$ vertices,
\begin{equation}
  \label{eq:5.1}
|E(G)|-|E(G^{(2)})| \leq \left\lfloor
  \frac{n}{2}\right\rfloor.  
\end{equation}
Furthermore, equality holds if and only if $G$ is the vertex-disjoint union of complete bipartite graphs of balanced part-size
with at most one component having odd number of vertices, i.e.
\[G= K_{t_1,t_1}\cup K_{t_2,t_2} \cup \cdots \cup K_{t_k,t_k} \cup K_{\lfloor \frac{n}{2}\rfloor-\sum_{i=1}^kt_i, \lceil \frac{n}{2}\rceil-\sum_{i=1}^kt_i},\]
for some positive integers $t_1,t_2,\ldots, t_k$ satisfying $\sum_{i=1}^kt_i= \lfloor \frac{n}{2}\rfloor$.
\end{lemma}

\begin{proof}
First, we will show that Equation \eqref{eq:5.1} holds for any forest.
Let $G$ be a forest.
Since $G$ is a forest, if $\{a,b\}\in E(G^{(2)})$ then $a$ and $b$ have a unique common neighbor in $G$.
Furthermore, given any vertex $c\in V(G)$, it follows that any pair of neighbors of $c$ is in $E(G^{(2)})$.
Thus we have
\begin{align*}
|E(G)|-|E(G^{(2)})| &= \frac{1}{2}\sum_{v\in V(G)} \deg(v) - \sum_{v\in
  V(G)} \binom{\deg(v)}{2} \\ 
            &=\sum_{v\in V(G)} \frac{1}{2}\deg(v)-\binom{\deg(v)}{2} \\
            &=\sum_{v\in V(G)}-\frac{1}{2}\deg(v)^{2} + \deg(v) \\
	    &\leq \sum_{v\in V(G)} \frac{1}{2} \\
	    &=\frac{n}{2}.
\end{align*}
The inequality above comes from the fact that $-\frac{1}{2}x^{2}+x \leq \frac{1}{2}$, attaining its maximum when $x=1$.
Since $|E(G)|-|E(G^{(2)})|$ is an integer, we have that 
\begin{equation}
  \label{eq:5.2}
|E(G)|-|E(G^{(2)})|\leq \left\lfloor \frac{n}{2}\right\rfloor  
\end{equation}
 as claimed.

Now we will prove the statement $|E(G)|-|E(G^{(2)})|\leq\left\lfloor \frac{|V(G)|}{2}\right\rfloor$ for general graphs using induction on the number of vertices.
It holds trivially for $n=1,2$.

Assume that the statement holds for all graphs with at fewer than $n$ vertices.
Consider a graph $G$ with $n$ vertices.
If $G$ is a forest, then the statement holds.
Otherwise, $G$ contains a cycle.
Choose $C_g$ to be a minimal cycle in $G$, i.e. one with no chords.
If $G=C_g$, then $\E(C_g)-\E(C_g^{(2)})=0$ if $g\not=4$ or $2$ if
$g=4$. The statement holds.

Now assume  $V(C)\subsetneq V(G)$.
Let $V_1:=V(C)=\{x_1,x_2,...,x_{g}\}$, where $x_{i}$ is adjacent to
$x_{i+1}$, and let $V_2:=V(G)\setminus V_1=\{v_1,v_2,...,v_{n-g}\}$.

The edges of $G$ can be partitioned into three parts: the induced graph $G[V_1]=C_g$, the induced graph $G[V_2]$, and the bipartite graph $G[V_1,V_2]$. 
Similarly, the edges of $G^{(2)}$ can be partitioned into three parts: $G^{(2)}[V_1]$, the induced graph $G^{(2)}[V_2]$, and the bipartite graph
$G^{(2)}[V_1,V_2]$. 
Now we compare term by term.

\begin{enumerate}
\item Note $|E(G^{(2)}[V_1])|\geq |E(C_g^{(2)})|$, and $|E(C_g^{(2)})|=g$ if $g\neq 4$ or $2$.
We have
\begin{equation} \label{eq:5.3}
	|E(G[V_1])|-|E(G^{(2)}[V_1])|\leq |E(C_g)|-|E(C_g^{(2)})|\leq \left\lfloor \frac{g}{2}\right\rfloor.  
\end{equation}

 \item By inductive hypothesis, we have $|E(G[V_2])|-|E((G[V_2])^{(2)})|\leq \left\lfloor \frac{n-g}{2}\right\rfloor$. 
  Combining with the fact $|E(G^{(2)}[V_2])\geq |E((G[V_2])^{(2)})|$, we have
\begin{equation} \label{eq:5.4}
	|E(G[V_2]))|-|E(G^{(2)}[V_2])|\leq \left\lfloor \frac{n-g}{2}\right\rfloor.    
\end{equation}
 
\item We claim $|E(G[V_1,V_2])| \leq |E(G^{(2)}[V_1,V_2])|$.
We define a map \[f\colon E(G[V_1,V_2]) \to E(G^{(2)}[V_1,V_2])\] as follows.
For any edge $x_iv\in E(G)$ with $v\in V_2$ and $x_i\in V_1$, define $f(vx_i)=vx_{i+1}$ (with the convention $x_{g+1}=x_1$).
Since $x_iv\in E(G)$ and $x_ix_{i+1}\in E(G)$, we have $vx_{i+1}\in E(G^{(2)})$.
The map $f$ is well-defined. 
We also observe that $f$ is an injective map. 
Thus 
\begin{equation}
  \label{eq:5.5}
|E(G[V_1,V_2])| \leq |E(G^{(2)}[V_1,V_2])|.  
\end{equation}
\end{enumerate}
Combining equations \eqref{eq:5.3}, \eqref{eq:5.4},  and \eqref{eq:5.5}, we get 
\[|E(G)|-|E(G^{(2)})|\leq \left\lfloor\frac{n-g}{2}\right\rfloor + \left\lfloor\frac{g}{2}\right\rfloor \leq \left\lfloor\frac{n}{2}\right\rfloor.\]

The inductive step is finished.

Now we check when equality holds.
It is straightforward to verify the sufficient condition; we omit the computation here.

Now we prove the necessary condition.
Assume that $G$ has $k+1$ connected components $G_1, G_2,\ldots, G_{k+1}$. 
Then we have
\[|E(G)|-|E(G^{(2)})|\leq \sum_{i=1}^{k+1} (|E(G_i)|-|E(G^{(2)}_i)|) \leq \sum_{i=1}^{k+1} \left\lfloor \frac{|V(G_i)|}{2} \right\rfloor \leq \left\lfloor \frac{n}{2}\right\rfloor.\]
If equality holds, then all but possibly one component has an even number of vertices.
It remains to show each component is a balanced complete bipartite graph.

Without loss of generality, we assume $G$ is connected.
If $G$ is a tree, then equality in Equation \eqref{eq:5.2} either forces the degree of every vertex to be $1$, or all the degrees are $1$ with a single exceptional vertex of degree $2$.
Since $G$ is assumed to be connected, $G$ is either $P_2=K_{1,1}$ or $P_3=K_{1,2}$.

Suppose that $G$ contains cycles, and the equalities hold in Equations \eqref{eq:5.3}, \eqref{eq:5.4}, and \eqref{eq:5.5}.
First we show that $C_4$ is the only possible chordless cycle in $G$.
Suppose not; let $C_g$ ($g\not=4$) be a cordless cycle.
We have $|E(C_g)|-|E(C^{(2)}_g)|=0$; which contradicts the assumption that equality holds in Equation \eqref{eq:5.3}.
Thus $G$ is a bipartite graph. 
Furthermore,the equality in \eqref{eq:5.4} forces each vertex $v$ to be connected to at least $2$ vertices of $C_4$.
Hence $G$ is 2-connected.
Now $G$ must be a complete bipartite graph. 
Otherwise, say $uv$ is a nonedge crossing the partite sets.
Since $G$ is $2$-connected, there exists a cycle containing both $u$ and $v$.
Let $C$ be such a cycle with minimum length; $C$ is cordless but not a $C_4$.
Contradiction.
Finally we show $G=K_{st}$ is balanced.
Note that
\[|E(G)|-|E(G^{(2)})|=st-{s\choose 2}-{t\choose 2}=\frac{n}{2}-\frac{(s-t)^2}{2}\leq \left\lfloor \frac{n}{2}\right\rfloor.\]
The equality holds only if $|s-t|\leq 1$. 
So $G$ is balanced.
\end{proof}

{\bf Proof of Theorem \ref{t:sd}:} 
We will prove by contradiction.
Let $R:=R(H)=R(H')$ be the common set of  edge types of $H$ and $H'$.
Suppose that $H'$ is not degenerate, then $\pi(H')>|R|-1 +\epsilon$ for some $\epsilon>0$.
Thus, there exists an $n_0$ satisfying $\pi_n(H')> |R|-1 +\epsilon/2$ for any $n\geq n_0$.
Let $G_n^R$ be a $H'$-free hypergraph with $\pi_n(G)> |R|-1 +\epsilon/2$.
Define a new hypergraph $G'_n$ over the same vertex set of $G$ with a new edge set
$E(G'_n)=E(G_n)\setminus E(G_n^{2}) \cup E((G_n^{2})^{(2)})$.
The hypergraph $G'_n$ is obtained for $G_n$ by replacing all $2$-edges by the edges in its square graph while keeping other type of edges.
By Lemma \ref{l:square}, we have
\begin{equation}
  \label{eq:6}
  \pi_n(G'_n)\geq \pi_n(G) - \frac{\lfloor
    \frac{n}{2}\rfloor}{{n\choose 2}}
\geq |R|-1 +
\epsilon/2 -\frac{1}{n}.
\end{equation}

Suppose that $H$ has $t$ $2$-edges.
Since $H$ is degenerate, so is the blowup hypergraph $H(t+1)$.
For sufficiently large $n$, $G'_n$ contains a subhypergraph $H(t+1)$.
By the definition of $G'$, for every copy of $H\subseteq H(t+1)$ and every $2$-edge $u_iv_i$
(for $1\leq i\leq t$) of $H$, there exists a vertex $x_i:=x_i(u_i,v_i)$ satisfying $u_ix_i$ and $v_ix_i$ are $2$-edges of $G$.
Our goal is to force that $x_1,x_2,\ldots, x_t$ are distinct from the vertices of $H$ and from each other.
This can be done by a greedy algorithm.
Suppose that the vertices of $H$ are listed by $y_1,y_2,y_3,\ldots,$ and so on.
Each vertex has $y_i$ has $t+1$ copies in $H(t+1)$. 
For $i=1,2,3,\ldots$, select a vertex $y_i'$ from the $t+1$ copies of $y_i$ so that $y_i'$ is not the same vertex as $x_j(u_j,v_j)$
for some $2$-edge $u_jv_j$ where $u_j, v_j$ have been selected.
This is always possible since $H$ has only $t$ $2$-edges.
Thus, we found a copy of $H'$ as a subgraph of $G$. Contradiction!
\hfill $\square$

It remains an open question to classify all non-degenerate hypergraphs.

In the remainder of this section, we generalize the following theorem due to 
Erd\H{o}s \cite{Erdos64} on the Tur\'an density of complete $k$-partitite $k$-uniform hypergraphs.

{\bf Theorem (Erd\H{o}s \cite{Erdos64}):}
{\rm
  Let $K^{(k)}_k(s_1,\ldots, s_k)$ be the complete $k$-partite
  $k$-uniform hypergraph with partite sets of size $s_1,\ldots,
  s_k$. Then any $K^{(k)}_k (s_1, \ldots, s_k)$-free $r$-uniform hypergraph
can have at most $O(n^{k-\delta})$ edges, where $\delta = \left(\prod^{k-1}_{i=1}
    s_i\right)^{-1}$.
}

We have the following theorem.
\begin{theorem} \label{tflag}
Let $L(s_1,s_2,\ldots, s_{v(L)})$ be a blowup of a flag $L^R$,
we have
$$\pi_n(L(s_1,s_2,\ldots, s_{v(L)}))=r-1+O(n^{-\delta}),$$
where $\delta= \frac{\max\{s_i\colon 1\leq i\leq v(L)\}}{\prod_{i=1}^{v(L)}s_i}$.
\end{theorem}

Using the concept of $H$-density, we can say a lot more about avoiding
a blowup of any hypergraph $H$.

Given two hypergraphs $H$ and $G$ with the same edge-type $R(H)=R(G)$,
the  {\em density} of $H$ in $G$, denoted by $\mu_H(G)$,
is defined as the probability
that a random injective map $f\colon V(H)\to V(G)$ satisfies
$H\stackrel{f}{\hookrightarrow} G$ (i.e. $f$ maps $H$ to an ordered copy of $H$ in $G$).
We have the following theorem.

\begin{theorem}
  For a fixed hypergraph $H$ on $m$ vertices and $m$ positive integers
  $s_1,s_2,\ldots, s_m$, let $H(s_1,s_2,\ldots, s_m)$ be the blowup of
  $H$. For sufficiently large $n$ and any hypergraph $G$ on $n$
  vertices with edge type $R(G)=R(H)$, if $G$ contains no subgraph
  $H(s_1,s_2,\ldots, s_m)$, then $$\mu_H(G)=O(n^{-\delta}),$$
where $\delta= \frac{\max\{s_i\colon 1\leq i\leq m\}}{\prod_{i=1}^{m}s_i}$.
\end{theorem}
{\bf Proof:}  We will prove by contradiction.
We assume $\mu_H(G)\geq C n^{-\delta}$ for some constant $C$ to be chosen later.
By reordering the vertices of $H$, we can assume $s_1\leq s_2\leq \cdots \leq s_m$.
Without loss of generality, we assume $n$ is divisible by $m$. 
Consider a random $m$-partition of $V(G)=V_1\cup V_2\cup \cdots \cup V_{m}$ where each part has size $\frac{n}{m}$. 
For any $m$-set $S$ of $V(G)$ ,  we say $S$ is a {\em transversal} (with respect to this partition of $V(G)$)
if $S$ intersects each $V_i$ exactly once.
The probability that $S$ is transversal is given by $\frac{(\frac{n}{m})^{m}}{{n\choose m}m!}$.

We say, an ordered copy $f(H)$ is a {\em transversal} if $f(V(H))$ is a transversal.
By the definition of $\mu_H(G)$, $G$ contains $\mu_H(G){n\choose m}m!$ ordered copies of $H$.
Thus, the expected number of transversal ordered copies of $H$ is 
\[\mu_H(G)\left(\frac{n}{m}\right)^{m}\geq \frac{C}{m^m} n^{m-\delta}.\]
There exists a partition so that the number of 
crossing maximum chains in $E(H)$
is at least $Cm^{-\delta} n^{m-{\delta}}$.
Now we fix this partition $[n]=V_1\cup \cdots \cup V_m$.

For $t_i \in \{1, s_i\}$ with $i=1,2\ldots, m$, we would like to
estimate the number of monochromatic (ordered) copies in $H'$, denoted
by $f(t_1, t_2, \ldots, t_{m})$, of $H(t_1,\ldots, t_{m})$ so
that the first $t_1$ vertices in $V_{\tau_1}$, the second $t_2$
vertices in $V_{\tau_2}$, and so on.

{\bf Claim a:} For $0\leq l\leq m-1$, we have
$$f(s_{1},\ldots, s_{l}, 1,\ldots, 1)\geq 
\left(1+o(1)\right)
\frac{\left(Cn^{-\delta}\right)^{\prod_{j=1}^l s_{j}}} 
{\prod_{j=1}^l \left(s_{j}! \right)^{\prod_{u=j+1}^ls_{u}}}
\left(\frac{n}{m}\right)^{m -l+ \sum_{j=1}^ls_j}.
$$ 
We prove claim (a) by induction on $l\in [0,m]$.
For the initial case $l=0$, the claim is trivial since 
$f(1,1,\ldots, 1)$ counts the number 
of transversal ordered copies of $H$.
 We have
$$f(1,1,\ldots,1)\geq Cn^{-\delta} \left(\frac{n}{m}\right)^{m}.$$
The statement holds for $l=0$.
Now we assume claim (a) holds for $l>0$. 
Consider the case $l+1$, for some $l\geq 0$. 
For any $S\in {V_{1}\choose s_{1}}\times \cdots \times {V_{l}\choose s_{}}\times V_{l+2}\times \cdots \times V_{m}$,
let $d_S$ be the number of vertices $v$ in $V_{l+1}$ such that the induced subgraph
of $H'$ on $S\times \{v\}$ is $H(s_{1},\ldots, s_{l}, s_{l+1},\ldots, 1)$. 
We have
\begin{align}
\label{eq:si}
f(s_1,\ldots, s_l,1,1, \ldots, 1) &= \sum_{S}d_S; \\
f(s_1,\ldots, s_l, s_{l+1},1,\ldots,1)&=\sum_{S} {d_S\choose s_{l+1}}.  
\label{eq:sil}
\end{align}

Let $\bar d_l$ be the average of $d_S$. 
By equation \eqref{eq:si} and the inductive hypothesis, we have
\begin{equation}
  \label{eq:dl}
\bar d_l\geq \frac{\sum_{S} d_S}{(\frac{n}{m})^{m -l-1+ \sum_{j=1}^ls_j}}
\geq  \left(1+o(1)\right)
 \frac{(\frac{n}{m})\left(Cn^{-\delta}\right)^{\prod_{j=1}^l s_j}} 
{\prod_{j=1}^l \left(s_j!\right)^{\prod_{u=j+1}^ls_u}}.
\end{equation}
Applying the convex inequality, we have
\begin{align*}
f(s_1,\ldots, s_l, s_{l+1},1,\ldots,1) &=\sum_{S} {d_S\choose s_{l+1}} \\
                                       &\geq \left(\frac{n}{m}\right)^{m-1 +\sum_{j=1}^l(s_j-1)} {\bar d_l\choose s_{l+1}} \\
                                       &=\left(1+O\left(\frac{1}{\bar d_l}\right)\right)\frac{\bar d_l^{s_{l+1}}}{s_{l+1}!} \left(\frac{n}{m}\right)^{m +\sum_{j=1}^{l+1}(s_j-1)}.
\end{align*}
Combining with equation \eqref{eq:dl}, we get
\[f(s_1,\ldots, s_l, s_{l+1},1,\ldots,1) \geq \left(1+o(1)\right) \frac{\left(Cn^{-\delta}\right)^{\prod_{j=1}^{l+1} s_j}} 
{\prod_{j=1}^{l+1} \left(s_j!\right)^{\prod_{u=j+1}^{l+1}s_u}} \left(\frac{n}{m}\right)^{m -l-1+ \sum_{j=1}^{l+1}s_j}.\]
In the last step, we used the assumption $s_{m}:=\max\{s_i\colon 1\leq i\leq m\}$.

Applying claim (a) with $l=m-1$, we get
\begin{align} \nonumber
f(s_1,s_2,\ldots, s_{m-1},1) &\geq
\left(1+o(1)\right)
\frac{\left(Cn^{-\delta}\right)^{\prod_{j=1}^{m-1} s_j}} 
{\prod_{j=1}^{m-1} \left(s_j!\right)^{\prod_{u=j+1}^{m-1}s_u}}
\left(\frac{n}{m}\right)^{1+ \sum_{j=1}^{m-1}s_j} \\
&=\left(1+o(1)\right) \frac{m^{-1}C^{\prod_{j=1}^{m-1} s_j} (\frac{n}{m})^{\sum_{j=1}^{m-1}s_j}}
{\prod_{j=1}^{m-1} \left(s_j! \right)^{\prod_{u=j+1}^{m-1}s_u}}.
\label{eq:lb}
\end{align}

For any $S\in {V_1\choose s_1}\times \cdots \times {V_{m-1}\choose s_{m-1}}$, 
let $d_S$ be the 
number of vertices $v$ in $V_{l+1}$ such that the edges in the induced subgraph
of $H'$ on $S\times \{v\}$ are monochromatic. Since $H'$ contains no
monochromatic copy of $C^R(s_1,\ldots, s_{m})$, we have
$d_S\leq s_{m}$. It implies
\begin{equation}
  \label{eq:ub}
f(s_1,s_2,\ldots, s_{m-1},1)=\sum_S d_S\leq s_{m} \left(\frac{n}{m}\right)^{\sum_{j=1}^{m-1}s_j}.  
\end{equation}

Choosing $C$ so that $C>\left((ms_{m})^{\frac{1}{\prod_{u=1}^{m}s_u}}\right)\cdot
\prod_{j=1}^{m-1} \left(s_j!\right)^{\frac{1}{\prod_{u=1}^{j}s_u}}$,
equations \eqref{eq:lb} and \eqref{eq:ub} contradict each other.
\hfill $\square$

{\bf Proof of Theorem \ref{tflag}:}
Consider a hypergraph $G:=G^R_n$ with \[h_n(G)=r-1+Cn^{-\delta}.\]
Let $r=|R|$.  
It suffices to show that $\mu_H(G)\geq C'n^{-\delta}$.

Given a random permutation $\sigma$,
let $X$ be the number of edges on a random full chain $\sigma(L)$.
By the definition of the Lubell function, we have $h_n(G)=\E(X)$.
Note $X$ only takes integer values $0,1,\ldots, r$.
Since $\E(X)>r-1$, there is non-zero probability that $X=r$.
In fact, we have
\begin{align*}
\E(X)&=\sum_{i=0}^r i\Pr(X=i)\\
&\leq r\Pr(X=r) + (r-1)(1-\Pr(X=r))\\
&=r-1 + \Pr(X=r).  
\end{align*}

Thus, we get 
\begin{equation}
  \label{eq:xr}
  \Pr(X=r)\geq \frac{C}{n^{\delta}}.
\end{equation}

Every flag $\sigma(L)$ contributes an equal share of the probability of the event that $X=r$,
namely,
\begin{equation}
  \label{eq:pchain}
  \frac{|Aut(L)|}{{n\choose v(L)}v(L)!}.
\end{equation}
Here $Aut(L)$ is the automorphism of $L$.
Thus, the number of such flags is at least
\begin{equation}
  \label{eq:nchain}
\frac{C}{|Aut(L)|n^{\delta}}{n\choose v(L)} v(L)!.
\end{equation}
It follows that
$\mu_H(G)\geq Cn^{-\delta}$.
\hfill $\square$

\section{Suspensions}
\begin{definition}
The {\em suspension} of a hypergraph $H$, denoted $S(H)$,
is the hypergraph with $V=V(H) \bigcup \{\ast\}$ where $\{\ast\}$ is an element not in $V(H)$,
and edge set $E=\{F\bigcup \{\ast\}\colon F\in E(H)\}$.
We write $S^{t}(H)$ to denote the hypergraph obtained by iterating the suspension operation $t$-times, i.e. $S^{2}(H)=S(S(H))$ and $S^{3}(H)=S(S(S(H)))$, etc.
\end{definition}
%
%
In this section we will investigate the relationship between $\pi(H)$ and $\pi(S(H))$ and look at limits such as $\lim_{t\to \infty}\pi(S^{t}(H))$.

\begin{definition}
Given a graph $G$ with vertex set $v_1,...,v_n$ the {\em link} hypergraph $G^{v_i}$ is the hypergraph with vertex set $V(G)\setminus \{v_i\}$ and
edge set $E=\{F\setminus \{v_i\}\colon v_i\in F \text{ and } F\in E(G)\}$.
\end{definition}

\begin{prop}
For any hypergraph $H$ we have that $\pi(S(H))\leq \pi(H)$.
\end{prop}

\begin{proof}
Let $G_{n}$ be a graph on $n$ vertices containing no copy of $S(H)$ such that $h_{n}(G_{n})=\pi_{n}(S(H))$.
Say $V(G_{n})=\{v_1,v_2,...,v_{n}\}$.  
Note that for any $v_i\in V(G_n)$, we have that Lubell value of the corresponding link graph is
\[h_{n-1}(G^{v_i}_{n})=\sum_{F\in G_n, v_i\in F} \frac{1}{\binom{n-1}{|F|-1}}.\]
Also, note that $G^{v_i}_n$ contains no copy of $H$.  If it did, then $S(H)\subset S(G^{v_i}_n)\subseteq G_n$; but $S(H)$ is not contained in $G_n$.
Thus $h_{n-1}(G^{v_i}_{n})\leq \pi_{n-1}(H)$.
We then have the following:
\begin{align*}
\pi_{n}(S(H)) &= h_n(G_n) \\
              &=\sum_{F\in E(G_n)} \frac{1}{\binom{n}{|F|}} \\
	      &=\sum_{F\in E(G_n)} \frac{1}{|F|} \sum_{v_i\in F} \frac{1}{\binom{n}{|F|}} \\
	      &=\sum_{i=1}^{n} \sum_{F\in E(G_n), v_i\in F} \frac{1}{|F|}\cdot \frac{1}{\binom{n}{|F|}} \\
	      &=\sum_{i=1}^{n} \sum_{F\in E(G_n), v_i\in F} \frac{1}{n}\cdot \frac{1}{\binom{n-1}{|F|-1}} \\
	      &=\frac{1}{n} \sum_{i=1}^{n} h_{n-1}(G^{v_i}_{n}) \\
	      &\leq \frac{1}{n} \sum_{i=1}^{n} \pi_{n-1}(H) \\
	      &=\pi_{n-1}(H).
\end{align*}
Thus, for any $n$, $\pi_{n}(S(H))\leq \pi_{n-1}(H)$; taking the limit as $n\to\infty$ we get the result as claimed.
\end{proof}

\begin{corollary}
  If $H$ is degenerate, so is $S(H)$.
\end{corollary}

\begin{conjecture}
For all $H$, $\displaystyle \lim_{t\to\infty}\pi(S^{t}(H))=|R(H)|-1$.
\end{conjecture}

To conclude our paper, we prove a special case of this conjecture. 

\begin{theorem}\label{t8}
Suppose that $H$ is a subgraph of the blowup of a chain.
Let $k_1$ be the minimum number in $R(H)$.
Suppose $k_1\geq 2$, and $H'$ is a new hypergraph obtained by adding finitely many edges of type
$k_1-1$ arbitrarily to $H$.
Then \[\lim_{t\to\infty} \pi(S^{t}(H'))=|R(H')|-1.\]
\end{theorem}

\begin{proof}
Without loss of generality, we can assume that $H$ is a blowup of a chain  and $V(H')=V(H)$.
(This can be done by taking blowup of $H$ and adding more edges.)

Suppose that $H$ has $v$ vertices and its edge type is $R(H):=\{k_1,k_2,\ldots, k_r\}$.
Set $k_0:=k_1-1$ so that $R(H'):=\{k_0,k_1,\ldots, k_r\}$.
For convenience, we write $R$ for $R(H)$ and $R'$ for $R(H')$, and
\begin{align*}
  R+t &:=\{k_1+t,k_2+t,\ldots, k_r+t\},\\
  R'+t &:=\{k_0+t,k_1+t,\ldots, k_r+t\}.
\end{align*}

For any small $\epsilon>0$, let $n_0=\lfloor \epsilon^{-t}\rfloor$.
For any $n\geq n_0$ and any hypergraph $G_n^{R+t}$ with \[\pi_n(G)>|R(H)|-1+\epsilon=r+\epsilon,\]
we will show $G$ contains a subhypergraph $S^t(H')$.

Take a random permutation $\sigma\in S_n$ and
let $X$ be the number of edges in $G$ hit by the random full chain 
$C_\sigma$:
$$\emptyset \subset \{\sigma(1)\} \subset \{\sigma(1),\sigma(2)\}
\cdots \subset \{\sigma(1), \sigma(2),\ldots, \sigma(i)\}
\subset \cdots \subset [n].$$
We have
$$\E(X)=\pi_n(G)>r+\epsilon.$$
Since $X\leq r+1$, we have
$$\E(X)=\sum_{i=0}^{r+1}i \Pr(X=i)\leq (r+1)\Pr(X=r+1) +r.$$
Thus, we get
\begin{equation}
  \label{eq:X=r+1}
\Pr(X=r+1)\geq \frac{\E(X)-r}{r+1}>\frac{\epsilon}{r+1}.  
\end{equation}

Recall that the density $\mu_H(G)$ is
the probability that a random injective map
$f\colon V(H)\to V(G)$ such that 
$H\stackrel{f}{\hookrightarrow}G$. Applying to $H=C^{R'+t}$,
we have
$$\mu_{C^{R'+t}}(G)=\Pr(X=r+1)>\frac{\epsilon}{r+1}. $$
Every copy of the chain $C^{R'+t}$ will pass through a set $A_1\in E^{k_1+t}(G)$.
Let $\mu_{C^{R+t}, A_1}(G)$ be the conditional probability
that a random injective map $f\colon V(C^{R+t}) \to V(G)$ satisfies
$C^{R+t}\stackrel{f}{\hookrightarrow}G$ given that the chain $C^{R+t}$ passes through $A_1$.
Let $d_-(A_1)$ be the number of sets $A_0$ satisfying
$A_0\in E^{k_0+t}(G)$ and $A_0\subset A_1$.
Then, we have
$$\mu_{C^{R'+t}}(G)=\frac{1}{{n\choose k_1+t}}\sum_{A_1\in
  E^{k_1+t(G)}}\mu_{C^{R+t}, A_1}(G)
\cdot \frac{d_-(A_1)}{k_1+t}.$$
Setting $\eta=\frac{\epsilon}{2(r+1)}$,
define a family 
$$\A=\{A_1\in E^{k_1+t}(G) \colon \mu_{C^{R+t}, A_1}(G)> \eta \mbox{ and } d_-(A_{1})> \eta (k_1+t)\}.$$
We claim $|\A|>\eta {n\choose k_1+t}.$ Otherwise,
we have
\begin{align*}
  \mu_{C^{R'+t}}(G)&=\frac{1}{{n\choose k_1+t}}\sum_{A_1\in
  E^{k_1+t(G)}}\mu_{C^{R+t}, A_1}(G)
\cdot \frac{d_-(A_1)}{k_1+t}\\
&= \frac{1}{{n\choose k_1+t}}\sum_{A_1\in\A}\mu_{C^{R+t}, A_1}(G)
\cdot \frac{d_-(A_1)}{k_1+t} +  \frac{1}{{n\choose k_1+t}}\sum_{A_1\not\in\A}\mu_{C^{R+t}, A_1}(G)
\cdot \frac{d_-(A_1)}{k_1+t}\\
&\leq\eta+\eta<\frac{\epsilon}{t+1}.
\end{align*}
Contradiction!

A $k_1$-configuration is a pair $(S,A_1)$ satisfying
$A_1\in \A$, $S=A_1\setminus \{i_1,i_2,\ldots, i_{k_1}\}$,
and $A_1\setminus \{i_j\}\in E^{k_0+t}(G)$ for any $1\leq j\leq k_1$.

For any $A_1\in \A$, the number of $S$ such that 
$(S,A_1)$ forms a $k_1$-configuration is at least
$${d_-(A_1)\choose k_1}\geq {\eta(k_1+t)\choose k_1}
>\left(\frac{\eta}{2}\right)^{k_1} {k_1+t\choose k_1}.$$
In the above inequality, we use the assumption $t>\frac{2}{\eta}k_1$.

By an averaging argument, there exists an $S$ so that
the number of $k_1$-configurations $(S, \bullet)$ is 
at least
$$\frac{|\A|\left(\frac{\eta}{2}\right)^{k_1} {k_1+t\choose k_1}}
{{n\choose t}}
\geq \frac{\eta^{k_1+1}}{2^{k_1}}{{n-t\choose k_1}}.$$
Now consider the link graph $G^S$. The inequality above
implies 
$$\mu_{C^R}(G^S)\geq \frac{\eta^{k_1+2}}{2^{k_1}}.$$
This implies $G^S$ contains a blow up of $C^R$.
Thus $G^S$ has a subhypergraph $H$. By the definition of
$k_1$-configuration, this $H$ can be extended to $H'$ in $G^S$.
In another words,
$G$ contains $S^t(H')$.
\end{proof}



\section{Connections to extremal poset problems}
As stated earlier, the Tur\'an density of non-uniform hypergraphs is
motivated by the extremal subset/poset problems. 

Let $\B_n=(2^{[n]},\subseteq)$ be the $n$-dimensional Boolean lattice.
Under the partial relation $\subseteq$, any family $\F\subseteq 2^{[n]}$
can be viewed as a subposet of $\B_n$.

For posets $P=(P,\le)$ and $P'=(P',\le')$, we say $P'$ is a {\em
weak subposet of $P$} if there exists an injection $f\colon P'\to P$
that preserves the partial ordering, meaning that whenever $u\le'
v$ in $P'$, we have $f(u)\le f(v)$ in $P$.
If $P'$ is not a weak poset of $P$, we say $P$ is $P'$-free. 
The following problems originate from 
Sperner's theorem, which states that the largest antichain
of $\B_n$ is $\nchn$.

{\bf Extremal poset problems:}
Given a fixed poset $P$, what is the largest size of a $P$-free family
$\F\subset \B_n$?

Let $\La(n,P)$ be the largest size of a $P$-free family $\F\subseteq \B_n$.
The value of $\La(n,P)$ is known for only a few posets $P$.
Let $\Pa_k$ be the (poset) chain of size $k$.
Then  $\La(n,\Pa_2)=\nchn$ by Sperner's theorem.
Erd\H{o}s \cite{Erdos45}  proved that
$\La(n,\Pa_k)=\Sigma(n,k)$, where $\Sigma(n,k)$ is the sum of
$k$ largest binomial coefficients. De Boinis-Katona-Swanepoel
\cite{DebKatSwa} proved $\La(n,\Oh_{4})=\Sigma(n,2)$. Here
$\Oh_4$ is the butterfly poset ($A,B \; \subset \; C,D$), or the crown poset of size $4$.

The asymptotic value of $\La(n, P)$ has been discovered for
various posets (see Table \ref{tab:1}). Let $e(P)$ be the largest
integer $k$ so that the family of $k$ middle layers of $\B_n$ is
$P$-free. Griggs and Lu \cite{GriLu} first conjecture
$\lim_{n\to\infty}\frac{\La(n,P)}{\nchn}$ exists and is an integer,
and it slowly involves into the following conjecture.

\begin{conjecture}\label{conj:2}
  For any fixed poset $P$, $\lim_{n\to\infty}\frac{\La(n,P)}{\nchn}=e(p)$.
\end{conjecture}

We overload the notation $\pi(P)$ for the limit
$\lim_{n\to\infty}\frac{\La(n,P)}{\nchn}$, where $P$ is a poset.
The conjecture is based on the observation of several previous known
results, which are obtained by Katona and others \cite{CarKat, DebKat,
DebKatSwa, GriKat, KNS, KatTar, Tha}. We summarize the known poset
$P$, for which the conjecture has been verified in Table \ref{tab:1}.

\begin{table}[hbt]
  \centering
  \begin{tabular}{|l|c|c|c|c|}
   \hline
$P$ & meaning & 
\parbox{1.4cm}{Hasse\\ Diagram} & $e(P)$ &References\\
\hline
\parbox{1cm}{fork
$\V_r$} & $A<B_1,\ldots,B_r$ &
\raisebox{-3mm}[5mm][2mm]{
\begin{tikzpicture}[scale=0.7, vertex/.style = {shape = circle,fill =
    black, minimum size = 2pt, inner sep=0pt}]
\node at (0,0) [vertex] (v0) {};
\node at (-0.5,1) [vertex] (v1) {};
\node at (-0.3,1) [vertex] (v2) {};
\node at (0.3,1) [vertex] (v4) {};
\node at (0.5,1) [vertex] (v5) {};
\draw (v0) -- (v1);
\draw (v0) -- (v2);
\draw [dotted] (v2) -- (v4);
\draw (v0) -- (v4);
\draw (v0) -- (v5);
\end{tikzpicture}
}
&1& \parbox{2.7cm}{
\cite{KatTar} for $r=2$;\\
\cite{DebKat} for general $r$.  
}\\
\hline
$\N$ & $A<B$, $B>C$, and $C<D$. &
\raisebox{-1mm}[5mm][2mm]{
\begin{tikzpicture}[scale=0.5, vertex/.style = {shape = circle,fill =
    black, minimum size = 2pt, inner sep=0pt}]
\node at (0,0) [vertex] (v0) {};
\node at (0,1) [vertex] (v1) {};
\node at (1,0) [vertex] (v2) {};
\node at (1,1) [vertex] (v3) {};
\draw (v0) -- (v1);
\draw (v1) -- (v2);
\draw (v2) -- (v3);
\end{tikzpicture}
}
&1& \parbox{2cm}{
\cite{GriKat}
}\\
\hline
\parbox{1.5cm}{butterfly
$\Oh_4$} & $A<B$, $B>C$, $C<D$, and $A<D$. &
\raisebox{-1mm}[5mm][2mm]{
\begin{tikzpicture}[scale=0.5, vertex/.style = {shape = circle,fill =
    black, minimum size = 2pt, inner sep=0pt}]
\node at (0,0) [vertex] (v0) {};
\node at (0,1) [vertex] (v1) {};
\node at (1,0) [vertex] (v2) {};
\node at (1,1) [vertex] (v3) {};
\draw (v0) -- (v1);
\draw (v1) -- (v2);
\draw (v2) -- (v3);
\draw (v0) -- (v3);
\end{tikzpicture}
}
&2& \parbox{2cm}{
\cite{DebKatSwa}
}\\
\hline
\parbox{1.4cm}{diamonds $\D_r$} & \parbox{4.5cm}{$A<B_1,\ldots,B_r<C$, where $r=3,4,7,8,9,15,16,\ldots$.
}
 &
\raisebox{-3mm}[6mm][3mm]{
\begin{tikzpicture}[scale=0.4, vertex/.style = {shape = circle,fill =
    black, minimum size = 2pt, inner sep=0pt}]
\node at (0,0) [vertex] (v0) {};
\node at (-0.5,1) [vertex] (v1) {};
\node at (-0.3,1) [vertex] (v2) {};
\node at (0.3,1) [vertex] (v3) {};
\node at (0.5,1) [vertex] (v4) {};
\node at (0,2) [vertex] (v5) {};
\draw (v0) -- (v1);
\draw (v0) -- (v2);
\draw (v0) -- (v3);
\draw (v0) -- (v4);
\draw (v5) -- (v1);
\draw (v5) -- (v2);
\draw (v5) -- (v3);
\draw (v5) -- (v4);
\end{tikzpicture}
}
&2& \parbox{2cm}{
\cite{GriLiLu}
}\\
\hline
\parbox{1cm}{baton
$P_k(s,t)$} &
\parbox{4.5cm}{ $A_1,\ldots, A_s<B_1<B_2<\cdots$
$\cdots<B_{k-2}<C_1,\ldots, C_t$}
 &
\raisebox{-4mm}[8mm][5mm]{
\begin{tikzpicture}[scale=0.5, vertex/.style = {shape = circle,fill =
    black, minimum size = 2pt, inner sep=0pt}]
\node at (0,0) [vertex] (b0) {};
\node at (0,0.5) [vertex] (b1) {};
\node at (0,1) [vertex] (b2) {};
\node at (-0.5,-0.5) [vertex] (a1) {};
\node at (-0.3,-0.5) [vertex] (a2) {};
\node at (0.3,-0.5) [vertex] (a3) {};
\node at (0.5,-0.5) [vertex] (a4) {};
\node at (-0.5,1.5) [vertex] (c1) {};
\node at (-0.3,1.5) [vertex] (c2) {};
\node at (0.3,1.5) [vertex] (c3) {};
\node at (0.5,1.5) [vertex] (c4) {};
\draw (b0) -- (b1);
\draw [dotted] (b1) -- (b2);
\draw (b0) -- (a1);
\draw (b0) -- (a2);
\draw (b0) -- (a3);
\draw (b0) -- (a4);
\draw (b2) -- (c1);
\draw (b2) -- (c2);
\draw (b2) -- (c3);
\draw (b2) -- (c4);
\draw [dotted] (c2) -- (c3);
\draw [dotted] (a2) -- (a3);
\end{tikzpicture}
}
&k-1& \parbox{3.3cm}{
\cite{Tha} for $s\!=\!1$, $k,t\!\geq\! 2$;\\
\cite{GriLu} for $k,s,t\geq 2.$  
}\\
\hline
\parbox{1cm}{tree $T$} & \parbox{4.5cm}{A tree poset, whose Hasse diagram is a tree.
Let $h(T)$ be the height of $T$.} &
\raisebox{-3mm}[5mm][2mm]{
\begin{tikzpicture}[scale=0.4, vertex/.style = {shape = circle,fill =
    black, minimum size = 2pt, inner sep=0pt}]
\node at (0,0) [vertex] (b0) {};
\node at (1,0) [vertex] (b1) {};
\node at (2,0) [vertex] (b2) {};
\node at (-0.5,1) [vertex] (c0) {};
\node at (0.5,1) [vertex] (c1) {};
\node at (1.5,1) [vertex] (c2) {};
\node at (-0.5,-1) [vertex] (a0) {};
\node at (0.5,-1) [vertex] (a1) {};
\node at (1.5,-1) [vertex] (a2) {};
\draw (a0) -- (b0);
\draw (a1) -- (b1);
\draw (a2) -- (b2);
\draw (b0) -- (c0);
\draw (b0) -- (c1);
\draw (b1) -- (c1);
\draw (b1) -- (c2);
\draw (a2) -- (c1);
\end{tikzpicture}
}
&$h(T)-1$& \parbox{3cm}{
\cite{GriLu} for $h(T)=2$\\
\cite{Buk}  for general cases.
}\\
\hline
\parbox{1cm}{Crowns $\Oh_{2t}$} & \parbox{4.5cm}{A height-2 poset,
  whose Hasse diagram is a cycle $C_{2t}$.} &
\raisebox{-3mm}[6mm][4mm]{
\begin{tikzpicture}[scale=0.75, vertex/.style = {shape = circle,fill =
    black, minimum size = 2pt, inner sep=0pt}]
\node at (0,0) [vertex] (b0) {};
\node at (0.2,0) [vertex] (b1) {};
\node at (0.4,0) [vertex] (b2) {};
\node at (0.8,0) [vertex] (b4) {};
\node at (1,0) [vertex] (b5) {};
\node at (0,1) [vertex] (c0) {};
\node at (0.2,1) [vertex] (c1) {};
\node at (0.4,1) [vertex] (c2) {};
\node at (0.8,1) [vertex] (c4) {};
\node at (1,1) [vertex] (c5) {};
\draw (b0) -- (c0);
\draw (b1) -- (c0);
\draw (b1) -- (c1);
\draw (b2) -- (c1);
\draw (b2) -- (c2);
\draw [dotted] (b2) -- (b4);
\draw [dotted] (c2) -- (c4);
\draw (b4) -- (c4);
\draw (b5) -- (c4);
\draw (b5) -- (c5);
\draw (b0) -- (c5);

\end{tikzpicture}
}
&$1$& \parbox{2.8cm}{
\cite{GriLu} for even $t\geq 2$;\\
\cite{crowns}  for odd $t\geq 7$.
}\\
\hline

  \end{tabular}
  \caption{Conjecture \ref{conj:2} has been verified for various posets $P$. }
  \label{tab:1}
\end{table}

The posets in Table \ref{tab:1} are far from complete. 
Let $\lambda_n(P)=\max\{h_n(\F)\colon \F\subseteq 2^{[n]}, P\mbox{-free}\}.$
A poset $P$ is  called {\em uniform-L-bounded} if $\lambda_n(P)\leq e(P)$
for all $n$. Griggs-Li \cite{GriLi, GriLi2} proved $\La(n,P)
=\sum_{i=\lfloor \frac{n+e(p)-1}{2}\rfloor}^{i=\lfloor
  \frac{n-e(p)+1}{2}\rfloor}{n\choose i}$
 if
$P$ is uniform-L-bounded. The uniform-L-bounded posets include
$\Pa_k$ (for any $k\geq 1$),  diamonds $\D_k$ (for $k\in [2^{m-1}-1,
2^m-\mchm-1]$ where  $m:=\lceil \log_2(k+2)\rceil$), and harps
$\aH(l_1,l_2,\ldots,l_k)$ (for $l_1>l_2>\cdots>l_k$), and other
posets. Noticeably,  Griggs-Li \cite{GriLi2} provides a method to
 construct large uniform-L-bounded posets from smaller
 uniform-L-bounded posets. There are infinitely many posets $P$ so that
$\pi(P)=e(P)$ holds.

Although there is no counter example found yet for Conjecture \ref{conj:2},
some posets have resisted efforts to determine their $\pi$ value.
The most studied, yet unsolved, poset is the diamond poset $\D_2$ (or $\B_2$,
$Q_2$ in some papers) as shown in Figure \ref{fig:d2c6c10}.
Griggs and Lu first observed $\pi(\D_2)\in [2,2.296]$.
Axenovich, Manske, and Martin~\cite{AxeManMar} came up with a new approach which improves the
upper bound to $2.283$. Griggs, Li, and Lu \cite{GriLiLu} further
improved the upper bound to $2.27\dot 3=2\frac{3}{11}$. 
Very recently, Kramer-Martin-Young \cite{KMY} recently proved $\pi(\D_2)\leq 2.25$.
While it seems to be hard to prove the conjecture $\pi(\D_2)=2$,
several groups of researchers have considered restricting the problem to three consecutive layers.
Let $\La^{c}(n,P)$ be the largest size a $P$-free family
$\F\subseteq\B_n$ such that $\F$ is in $e(p)+1$ consecutive layers.
Let $\pi^c(P)=\lim_{n\to\infty}\frac{\La^c(n,p)}{\nchn}$, if the limit
exists. Here is a weaker conjecture (of consecutive layers).

\begin{conjecture}\label{conj:3}
For any fixed poset $P$, $\pi^c(P)=e(p)$.
\end{conjecture}

\begin{figure}[hbt]
  \centering
\begin{tikzpicture}[scale=0.7, vertex/.style = {shape = circle,fill =
    black, minimum size = 2pt, inner sep=0pt}]
\node at (0,0) [vertex] (v0) {};
\node at (-0.5,1) [vertex] (v1) {};
\node at (0.5,1) [vertex] (v2) {};
\node at (0,2) [vertex] (v3) {};
\draw (v0) -- (v1);
\draw (v0) -- (v2);
\draw (v1) -- (v3);
\draw (v2) -- (v3);
\node at (0,-0.5) (label) {$\D_2$};
\end{tikzpicture}
\hfil
\begin{tikzpicture}[scale=0.7, vertex/.style = {shape = circle,fill =
    black, minimum size = 2pt, inner sep=0pt}]
\node at (0,0) [vertex] (b0) {};
\node at (1,0) [vertex] (b1) {};
\node at (2,0) [vertex] (b2) {};
\node at (0,2) [vertex] (c0) {};
\node at (1,2) [vertex] (c1) {};
\node at (2,2) [vertex] (c2) {};
\draw (b0) -- (c0);
\draw (b1) -- (c0);
\draw (b1) -- (c1);
\draw (b2) -- (c1);
\draw (b2) -- (c2);
\draw (b0) -- (c2);
\node at (1,-0.5) (label) {$\Oh_6$};
\end{tikzpicture}
\hfil
\begin{tikzpicture}[scale=0.7, vertex/.style = {shape = circle,fill =
    black, minimum size = 2pt, inner sep=0pt}]
\node at (0,0) [vertex] (b0) {};
\node at (1,0) [vertex] (b1) {};
\node at (2,0) [vertex] (b2) {};
\node at (3,0) [vertex] (b3) {};
\node at (4,0) [vertex] (b4) {};
\node at (0,2) [vertex] (c0) {};
\node at (1,2) [vertex] (c1) {};
\node at (2,2) [vertex] (c2) {};
\node at (3,2) [vertex] (c3) {};
\node at (4,2) [vertex] (c4) {};
\draw (b0) -- (c0);
\draw (b1) -- (c0);
\draw (b1) -- (c1);
\draw (b2) -- (c1);
\draw (b2) -- (c2);
\draw (b3) -- (c2);
\draw (b3) -- (c3);
\draw (b4) -- (c3);
\draw (b4) -- (c4);
\draw (b0) -- (c4);
\node at (2,-0.5) (label) {$\Oh_{10}$};
\end{tikzpicture}
  \caption{Three most wanted posets for Conjecture \ref{conj:2}: $\D_2$, $\Oh_6$, and $\Oh_{10}$. }
  \label{fig:d2c6c10}
\end{figure}
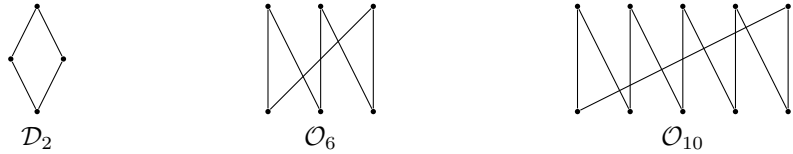

Axenovich-Manske-Martin \cite{AxeManMar} first proved $\pi^c(\D_2)\leq
2.207$; it was recently improved to 
 $2.1547$ (Manske-Shen
\cite{MS}) and $2.15121$ (Balogh-Hu-Lidick\'y-Liu \cite{BHLL}).

We say a hypergraph $H$ {\em represents} a poset $P$ if
the set of edges of $H$ (as a poset) is isomorphic to $P$.
For any fixed finite poset $P$, by the definition of $e(P)$,
there exists a hypergraph $H\subseteq \B_{n_0}$ with $|R(H)|=e(P)+1$ 
representing a superposet of $P$.

\begin{theorem}\label{t9}
Suppose that a hypergraph $H$ with $|R(H)|=e(P)+1$ represents
a superposet of $P$. Then, for any integer $t\geq 0$, we have
$$\pi^c(P)\leq \pi(S^t(H)).$$
\end{theorem}

{\bf Proof:}
Let $x:=\pi(S^t(H))$. For any $\epsilon>0$, 
there exists an $n_1$ so that $\pi_{n}(S^{t}(H))
\leq x+\epsilon$ for all $n\geq n_1$. We claim
$$\pi^c(P)\leq x+2\epsilon.$$
 Otherwise, for any sufficiently large $n$,
there exists a family $\F\subset \B_n$ which is in $e(p)+1$
consecutive layers with $|\F|>(x+\epsilon)\nchn$.
Let $k_0$ be the smallest size of edges in $S^{t}(H)$.
Let $k_1$ be the integer that $\F$ is in $k_1$-th to $(k_1+e(P))$-th
layer. Since $e(P)\leq x < e(P)+1$, we have
$$|\F|>(x+\epsilon)\nchn\geq (e(P)+\epsilon)\nchn.$$
Note any layer below $\frac{n}{2}-2\sqrt{n\ln n}$ can
only contribute $\frac{2^n}{n^2}$, which is less than $\epsilon \nchn$
for sufficiently large $n$. We get 
$$k_1\geq \frac{n}{2}-2\sqrt{n\ln n}.$$
Choose $n$ large enough so that $k_1\geq k_0$ and $n-k_1+k_0\geq n_1$.
We observe that 
$$h_n(\F)\geq  \frac{|F|}{\nchn}> x+2\epsilon.$$
By the property of Lubell function,
$h_n(\F)$ is the average of $h_{n+k_0-k_1}(\F_S)$
over all $S\in {[n]\choose k_1-k_0}$, where 
$\F_S$ is the link hypergraph over $S$.
Therefore, there exists a set $S\in {[n]\choose k_1-k_0}$ so that
$h_{n+k_0-k_1}(\F_S)>x+2\epsilon$.
Thus, $\F_S$ contains a subhypergraph  $S^{t_0}(H)$.
In particular, $\F$ contains a subposet $P$.

Thus, we have
$$\pi^c(P)\leq x+2\epsilon.$$
Since this holds for any $\epsilon>0$, we have
$\pi^c(P)\leq x.$
\hfill $\square$

\begin{corollary}
   Conjecture \ref{conj:2} implies Conjecture \ref{conj:3}. 
\end{corollary}

In particular, from Theorem \ref{t8}, we get a new family of posets $P$
so that $\pi^c(P)=e(p)$. A special example is the crown
$\Oh_{2t}$, where $t=4$ and $t\geq 6$. The idea can be traced back 
from Conlon's concept {\em $k$-representation} of bipartite graphs
\cite{Conlon}. Theorem \ref{t8} can be viewed as a natural
generalization of Conlon's theorem.
It is easy to generate more examples of
posets in this family. However, a complete
description of these posets is tedious; thus it is omitted here.

Note that the complete hypergraph $K_2^{\{0,1,2\}}$ has $4$ edges
$\emptyset, \{1\}, \{2\}, \{1,2\}$; which form the diamond poset
$\D_2$. In particular, for any $t\geq 0$, we have
$$\pi^c(\D_2)\leq \pi(S^t(K_2^{\{0,1,2\}})).$$
This provides a possible way to improve the bounds of $\pi^c(\D_2)$.

\end{document}